\documentclass[11pt,reqno]{amsart}
\usepackage{amsmath,amsfonts,amssymb,amscd,amsthm,amsbsy,bbm,epsf,calc,graphicx}
\usepackage{color}
\usepackage{datetime}
\usepackage{gensymb}
\usepackage{wasysym}
\usepackage[pdftex,bookmarks,colorlinks,breaklinks]{hyperref}

\numberwithin{equation}{section}

\newcounter{hours}\newcounter{minutes}

\theoremstyle{plain}
\newtheorem{thm}{Theorem}[section]

\newtheorem{lem}[thm]{Lemma}
\newtheorem{cor}[thm]{Corollary}
\newtheorem{prop}[thm]{Proposition}

\theoremstyle{definition}                  %% For unnumbered Remarks, etc.
\newtheorem{rem}[thm]{Remark}
\newtheorem{rems}{Remarks}

\def\tr{\textnormal{tr}}
\def\dive{\textnormal{div}}

\def\Id{\mathbb{I}}
\def\Tr{\textnormal{Tr}}

\def\maxw{\mathcal{M}} % Maxwellian
\def\fin{f_{\textnormal{in}}} % Initial data

\def\QL{Q_{\mathcal{L}}}
\def\QKS{Q_{\mathcal{KS}}}

\newcommand{\R}{\mathbb{R}}

	\title{Estimates for radial solutions of the homogeneous Landau equation with Coulomb potential}
    \author{Maria Gualdani and Nestor Guillen}

\begin{document}

	\begin{abstract}
		Motivated by the question of existence of global solutions, we obtain pointwise upper bounds for radially symmetric and monotone solutions to the homogeneous Landau equation with Coulomb potential. The estimates say that blow up in the $L^\infty(\R^3)$-norm at a finite time $T$ can occur only if the $L^{3/2}(\R^3)$-norm of the solution concentrates for times close to $T$.  The bounds are obtained using the comparison principle for the Landau equation and for the associated mass function. 
		
		This method provides long-time existence results for the isotropic version of the Landau equation with Coulomb potential, recently introduced by Krieger and Strain. 
		
			\end{abstract}

\maketitle
    \markboth{M. Gualdani, N. Guillen}{Estimates for radial solutions of the homogeneous Landau equation with Coulomb potential}

\section{Introduction}\label{sec: introduction}

 This manuscript is concerned with the Cauchy problem for the homogeneous Landau equation: such equation takes the general form
  \begin{align}\label{eqn: Landau_g}
         \partial_t f (v,t) = Q(f,f), \quad f(v,0)=f_{in}(v),\quad v\in \R^3, \quad t>0,
    \end{align}
    where $Q(f,f)$ is a quadratic operator known as the Landau collisional operator
        \begin{align}\label{kernel_landau}
Q(f,f)=\textrm{div}\left( \int_{\R^3} A(v-y)\left( f(y)\nabla_v f(v)-f(v)\nabla_y f(y)\right)dy\right).
\end{align} 
The term $A(v)$ denotes a positive and symmetric matrix   

    \begin{equation*}
         A(v) := \frac{1}{8\pi}\left(\textrm{Id} - \frac{v \otimes v}{|v|^2}\right) \varphi(|v|),\;\;\; v\neq 0,
    \end{equation*}
which acts as the projection operator onto the space orthogonal to the vector $v$. The function $\varphi(|v|)$ is a scalar valued function determined from the original Boltzmann kernel describing how particles interact. If the interaction strength between particles at a distance $r$ is proportional to ${r^{-s}}$, then 
\begin{align}\label{potential}
\varphi(|v|) := |v|^{\gamma+2},\quad \gamma = \frac{(s-5)}{(s-1)}.
\end{align}

Any solution to  \eqref{eqn: Landau_g}-\eqref{kernel_landau} is an integrable and nonnegative scalar field $f(v,t): \R^3 \times [0,T] \to \R^+$. Equation \eqref{eqn: Landau_g} describes the evolution of a plasma in spatially homogeneous regimes, which means that the density function $f$ depends only on the velocity component $v$.  Landau's original intent in deriving this approximation was to make sense of the Boltzmann collisional operator, which always diverges when considering purely grazing collisions.\\

The Cauchy problem for \eqref{eqn: Landau_g}-\eqref{potential} is very well understood for the case of hard potentials, which correspond to $\gamma\ge 0$ above. Desvillettes and Villani showed the existence of global classical solutions for hard potentials and studied its long time behavior, see \cite{DesVil2000a,DesVil2000b, Villani-Handbook02} and references therein. In this case there is a unique global smooth solution, which converges exponentially to an equilibrium distribution, known as the Maxwellian function 
\begin{align*}
\maxw(v) = \frac{1}{(2\pi)^{3/2}}e^{-\frac{|v|^2}{2}}.
\end{align*}

Analyzing the soft potentials case, $\gamma <0$, has proved to be more difficult. Using a probabilistic approach, Fournier and Guerin \cite{FG09} obtained uniqueness and existence of weak solutions for the case of moderately soft potentials ($\gamma \in [-2,0]$), uniqueness was also shown to hold for bounded solutions in the Coulomb case in \cite{Fou2010}. On the other hand, in recent work of Alexnadre, Liao and Lin \cite{ALL13} it is proved -for moderately soft potentials- that the $L^2$ of the solution remains bounded for finite times, as long as the initial data is in $L^2$. For $\gamma \in [-3,-2]$, short time existence is known (going back to work of Arsenev and Peskov \cite{AP77} for initial data in $L^\infty$, see also \cite{ALL13} ) as well as global existence under a smallness $L^2$ assumption on the initial configuration, as proved by Wu \cite{Wu13}. 
   In a different direction, Villani \cite{V98} introduced the so called $H$-solutions, which enjoy (weak) a priori bounds in a weighted Sobolev space. However, the issue of their uniqueness and regularity (i.e. no finite time break down occurs) has remained open, even for smooth initial data: see  \cite[Chapter 1, Chapter 5]{Villani-Handbook02} for further discussion.\\
   
   Guo in \cite{Guo02} employs a completely different approach based on perturbation theory for the existence of periodic solutions to the spatially inhomogeneous Landau equation in $\mathbb{R}^3$. He shows that if the initial data is sufficiently close to the unique equilibrium in a certain high Sobolev norm then a unique global solution exists. Moreover, as remarked in  \cite{Guo02}, this approach also extends to the case of potentials \eqref{potential} where $\gamma$ might even take values below $-3$.

{Due to the lack of a global well-posedness theory,  several conjectures about possible finite-time blow up for general initial data have been made throughout the years.} In  \cite{Villani-Handbook02} Villani discussed the possibility that \eqref{eqn: Landau_g}-\eqref{potential} could blow up for $\gamma =-3$. Note that for smooth solutions \eqref{eqn: Landau_g}-\eqref{potential} with $\gamma =-3$ can be rewritten as
\begin{align}\label{eqn: Landau}
\partial_t f = \textrm{div}( A[f]\nabla f -f \nabla a[f])= \Tr(A[f]D^2f) + f^2,
\end{align}
where 
\begin{align*}%\label{Aa}
A[f] := A(v)\ast f = \frac{1}{8\pi |v|}\left( \textrm{Id} - \frac{v \otimes v}{|v|^2}\right)\ast f , \quad \Delta a = -f.
\end{align*}

Equation \eqref{eqn: Landau} can be thought of as a quasi-linear nonlocal heat equation. Supports for blow-up conjectures were given by the fact that  
\eqref{eqn: Landau} is reminiscent of the well studied semilinear heat equation
    \begin{equation}\label{eqn: nonlinear heat equation}
         \partial_t f = \Delta f + f^2,	
    \end{equation}
Blow up for \eqref{eqn: nonlinear heat equation} in $L^\infty$ implies blow up of every $L^p$ norm with $p>3/2 $, as shown for instance in work of Giga and Kohn \cite{GigaKohn85}, where the singularities are studied using self-similar variables. 

However, despite the apparent similarities, equation \eqref{eqn: Landau} behaves differently than \eqref{eqn: nonlinear heat equation}. The Landau equation admits a richer class of equilibrium solution: every Maxwellian $\maxw$ solves $Q(\maxw,\maxw) = 0$ which holds, in particular, for those with arbitrarily large mass.\\

  From a different perspective, Krieger-Strain \cite{KriStr2012} considered an isotropic version of \eqref{eqn: Landau}
    \begin{equation}\label{eqn: Krieger Strain alpha}
         \partial_t f = a[f]\Delta f + \alpha f^2,\;\;\alpha\in [0,1],
    \end{equation}
    and showed global existence of smooth radial solutions starting from radial initial data when $\alpha<2/3$. This range for $\alpha$ later was expanded to any $\alpha<74/75$ by means of a non-local inequality obtained by Gressman, Krieger and Strain \cite{GreKriStr2012}. Note that when $\alpha=1$, the above equation can be written in divergence form,	
	\begin{equation}\label{eqn: Krieger Strain}
         \partial_t f = \textrm{div}( a[f]\nabla f -f \nabla a[f]).	
	\end{equation}
These results put in evidence how a non-linear equation with a non-local diffusivity such as \eqref{eqn: Krieger Strain} behaves drastically different (and better) than \eqref{eqn: nonlinear heat equation}. \\
    
 Our main results in this manuscript are twofold. The first one gives necessary conditions for the finite time blow up of solutions to \eqref{eqn: Landau}. The second (unconditional) result says that solutions to \eqref{eqn: Krieger Strain} do not blow up at all. Both theorems deal only with radially symmetric, decreasing initial conditions.
    
   On the initial condition $f_{in}$ we make the following assumption: for some $p>6$, $f_{in}$ satisfies
    \begin{equation}\label{eqn: initial assumptions}
    	  \begin{array}{c}
    	  	  \fin \in L^1(\mathbb{R}^3) \cap L^p_{\textrm{weak}}(\mathbb{R}^3) ,\quad f_{in}\ge 0, \quad \| f_{in}\|_{L^1(\R^3)} =:M_{in},\\
		                 \fin \textnormal{ radially symmetric, monotone decreasing.}       
    	  \end{array} 
    \end{equation}

  The main results are the following.

  \begin{thm}\label{thm: conditions for Landau blowup}
    Let $\fin$ be as in \eqref{eqn: initial assumptions}. Then there is a smooth radial solution $f(v,t)$ of \eqref{eqn: Landau} and of (\ref{eqn: Krieger Strain}) defined in $\R^3 \times [0,T_0)$ for a positive (possibly infinite) $T_0$. Moreover, $T_0$ is maximal in the sense that either $T_0 = \infty$ or the $L^{3/2}$ norm of $f$ accumulates at $T=T_0$, in particular
    \begin{equation*}
        \lim \limits_{t \to T_0^-} \|f(\cdot,t)\|_{L^{p}(B_1)} = \infty,\;\;\forall\; p>3/2.
    \end{equation*}
    
  \end{thm}
 
  This certainly does not yield long time existence of classical solutions to \eqref{eqn: Landau}. However, the ideas used in proving Theorem \ref{thm: conditions for Landau blowup} can at least guarantee long time existence for \eqref{eqn: Krieger Strain}. 

 \begin{thm}\label{thm: global existence KS}
   Let $\fin$ be as in \eqref{eqn: initial assumptions}, then there is $f:\mathbb{R}^3\times \mathbb{R}_+\to\mathbb{R}$, smooth for positive times, solving \eqref{eqn: Krieger Strain} and such that $f(\cdot,0)=f_{in}$.
 \end{thm}

  %Since for both equation \eqref{eqn: Krieger Strain} and \eqref{eqn: Landau} conservation of mass holds, the assumption $\| f_{in}\|_{L^1(\R^3)} =1$ is not a restriction; all our results trivially hold also for $\| f_{in}\|_{L^1(\R^3)} \neq 1$. 
    
%%%%%%%%%%%%%%%%%%%%%%%%%%%%%%%%%%%
%%%%%%%%%%%%%%%%%%%%%%%%%%%%%%%%%%%	
   % \subsection{Mathematics background}

We approach the analysis from the point of view of nonlinear parabolic equations. The nonlocal dependence of the coefficients on the solution {prevents the equation from satisfying a comparison principle} : if $v_0$ is a contact point of two functions $f$ and $g$, i.e. $f(v_0) = g(v_0)$ and everywhere else $f(v) < g(v)$, it does not hold that $Q(f,f)(v_0)\leq Q(g,g)(v_0)$. More precisely, one cannot expect an inequality such as
    \begin{equation*}
         \Tr(A[f]D^2f)(v_0) \leq \Tr(A[g]D^2g ) (v_0).	
    \end{equation*}
   In fact due to the nonlocality of $A$ one only has $A[f](v_0)\le A[g](v_0)$. Equality $A[f] (v_0) = A[g](v_0) $ holds only when $f\equiv g$ for every $v\in \R^3$. In addition, also {maximum principle does not hold}, since at a maximum point for $f$ we only obtain $\partial_t f\leq -f \Delta a[f]$, which does not rule out finite time blow up of the maximum of $f$. However, even if maximum and comparison principle fail, if one can control in some $L^p$ space the size of $a[f]$ and the ellipticity (i.e. eigenvalues) of the matrix $A[f]$, then higher regularity for $f$ via a bootstrapping effect can be shown. Thus, {the problem of regularity estimates for \eqref{eqn: Landau} becomes a question of bounding $f$ in some $L^p$ norm}, with $p>1$. \\
    
        A previous attempt by the authors that meant to cover a much more general situation (global existence for bounded, fast decaying initial data), was ultimately undone by a computational error. However we kept the main idea of barrier arguments to show  global existence results for (\ref{eqn: Krieger Strain}) and conditional existence for the Landau equation (\ref{eqn: Landau}).

 %%%%%%%%%%%%%%%%%%%%%%%%%%%%%%%%%%%
%%%%%%%%%%%%%%%%%%%%%%%%%%%%%%%%%%%

     \subsection{Outline} The rest of the paper is organized as follows. After a brief review in Section \ref{sec: basic properties linear equation} on nonlinear parabolic theory that will be needed to construct local solutions to the non-linear problems,  in Section \ref{sec: radial symmetry} we outline the symmetry properties of \eqref{eqn: Landau}. Section \ref{short_time_existence} deals with short time existence. In Section \ref{sec: contact analysis} we present a barrier argument that will allow to prove conditional non-blow up results for the Landau equation and global well-posedness for the Krieger-Strain equation in Section \ref{mass_comparison}.

    \subsection{Notation} Universal constants will be denoted by $c,c_0,c_1,C_0,C_1,C$. Vectors in $\mathbb{R}^3$ will be denoted by $v,w,x,y$ and so on, the inner product between $v$ and $w$ will be written $(v,w)$. $B_R(v_0)$ denotes closed ball of radius $R$ centered at $v_0$, if $v_0=0$ we simply write $B_R$. 
    The identity matrix will be noted by $\mathbb{I}$, the trace of a matrix $X$ will be denoted $\Tr(X)$.  The initial distribution for the Cauchy problem will always be denoted by $\fin$.

The letter $\Omega$ denotes a general compact subset of $\R^3$. $Q\subset \mathbb{R}^3\times \mathbb{R}_+$ is a space-time cylinder of parabolic diameter $R$ with $R>0$ a general constant, unless otherwise specified.   $\partial_p Q$ denotes the parabolic boundary of $Q$.

\subsection{Acknowledgements.} MPG is supported by NSF DMS-1310746 and DMS-1412748. NG is supported by NSF-DMS 1201413. The authors would like to thank MSRI for the hospitality during the program Free Boundary Problems, Theory and Applications in the Spring of 2011, where this work was started. We also would like to thank Luis Silvestre and Cedric Villani for many fruitful communications.  MPG would like to thank NCTS Mathematics Division Taipei for the kind hospitality.

\section{A rapid review of linear parabolic equations}\label{sec: basic properties linear equation}

	We will work with two bilinear operators, namely the one associated to the equation
	  \begin{equation*}
	    \QL(g,f) := \dive( A[g]\nabla f-f\nabla a[g]) = \tr[A[g]D^2f]+fg,
	  \end{equation*}
	  and the one associated to the equation of Krieger and Strain,
	  \begin{equation*}
	    \QKS(g,f) := \dive( a[g]\nabla f-f\nabla a[g]) = a[g]\Delta f+fg.
	  \end{equation*}  
    As it is well known, through $\QL$ (and also $\QKS$) any $g:\mathbb{R}^3\times \mathbb{R}_+ \to \mathbb{R}$, gives rise to a linear elliptic operator with variable coefficients, as follows:
    \begin{align*}
         \phi \to \QL(g,\phi) & :=\dive(A[g]\nabla \phi- \phi\nabla a[g])= \tr(A[g]D^2 \phi)+ \phi g,\\
         \phi \to \QKS(g,\phi) & :=\dive(a[g]\nabla \phi- \phi\nabla a[g])= a[g]\Delta \phi+ \phi g.		 
    \end{align*}
    Accordingly, given such a $g$ and initial data $\fin$, one considers the linear Cauchy problem,
    \begin{equation}\label{eqn: approximate Linear Cauchy problem}
	     \left \{ \begin{array}{rl}
	          \partial_t f & = Q(g,f) \text{ in } \mathbb{R}^3 \times \mathbb{R}_+,\\
	          f(\cdot,0) & =  \fin,
	     \end{array}\right.	
	\end{equation}
    both when $Q=\QL$ or $Q=\QKS$.
  \begin{rems}\label{decay_mass}
    Note that $\QL(g,f)$ and $\QKS(g,f)$ can both be expressed as a divergence, so any solution to \eqref{eqn: approximate Linear Cauchy problem} preserves its mass over time, i.e. $\|f(\cdot,t)\|_{L^1(\mathbb{R}^3)}=\|f_{in}(\cdot)\|_{L^1(\mathbb{R}^3)}=:M_{in}$ for all $t>0$. 
  \end{rems}

    \begin{lem}\label{lem: existence uniqueness regularity for linear problem}
        Let $\fin :\mathbb{R}^3\to\mathbb{R},\;\; g:\mathbb{R}^3\times \mathbb{R}_+ \to \mathbb{R}$ be non-negative functions such that 
        \begin{equation}\label{init_data_c_beta}
            \begin{array}{l}
                \fin \in L^1(\mathbb{R}^3) \cap C^{2,\beta}(\mathbb{R}^3),\\
                g \in L^\infty(\mathbb{R}_+, L^1(\mathbb{R}^3))\cap C^{\beta}(\mathbb{R}^3 \times \mathbb{R}_+),\;\;   \textnormal{ for some } \beta \in (0,1).    
            \end{array}
        \end{equation}
        For any $\delta>0$, there exists a unique $f:\mathbb{R}^3\times \mathbb{R} \to \mathbb{R}$ which is a classical solution of
        \begin{equation}
            \left \{ \begin{array}{rl}
                \partial_t f & = \delta \Delta f+Q(g,f) \textnormal{ in } \mathbb{R}^3\times \mathbb{R}_+,\\
                f(\cdot ,0) & = \fin,     
            \end{array}\right.
        \end{equation}
        where $Q(\cdot,\cdot)$ denotes either $Q=\QL$ or $Q=\QKS$. 
    \end{lem}
    
   Next recall several parabolic regularity estimates dealing with equations of the form
	\begin{align*}
	  \partial_tf  = \dive\left ( B\nabla f + f b \right ),	
	\end{align*}
	where $f: Q \to \mathbb{R}$ and $Q = B_R(x_0)\times (t_0-R^2,t_0) \subset \mathbb{R}^d\times \mathbb{R}$ is the parabolic cylinder of radius $R$ centered at some points $x_0,t_0$. The first two theorems are respectively a local H\"older estimate (De Giorgi-Nash-Moser) and a $L^\infty$ estimate for $f$ in terms of its boundary data (Stampacchia estimate), see \cite[Chapter III, Theorem 10.1, page 204]{LadSolUra1968} and \cite[Chapter IV, Theorem 10.1, page 351]{LadSolUra1968} as well as \cite[Chapter VI, Theorem 6.29 p. 131]{Lie1996} for the respective proofs. The main point of these theorems is that they do not require any regularity assumption on the diffusion matrix $B$ (beyond ellipticity and boundedness).

	\begin{thm}\label{thm: parabolic regularity} (De Giorgi-Nash-Moser estimate.)
	     Suppose $f$ is a weak solution of the equation 
	     \begin{align*}
	          \partial_tf  = \dive\left ( B\nabla f + f b \right ),	
	     \end{align*}
         where $b$ is a vector field and $B$ is a symmetric matrix such that
		 \begin{align*}
	       \lambda\; \Id \leq B(v,t) & \leq \Lambda\; \Id \text{ a.e. in } Q.
	     \end{align*}
	     Then, there is some $\alpha \in (0,1)$ and $C>0$ such that the following estimate holds:
         \begin{equation}\label{eqn: DeGiorgi Nash Moser estimate}
	       [f]_{C^{\alpha,\alpha/2}(Q_{1/2})}\leq C\left (\|f\|_{L^\infty(Q)}+\|g\|_{L^\infty(Q)}+R^2\|b\|_{L^\infty(Q)} \right ),
	     \end{equation}		 
		 where $Q_{1/2}:= B_{R/2}(x_0)\times (t_0-(R/2)^2,t_0)$ and $\alpha$ and $C$ are determined by $\lambda,\Lambda,R$ and $d$.
	\end{thm}

    \begin{thm}\label{thm: Stampacchia} (Stampacchia estimate.) 
        If $f$ is a weak solution of
        \begin{equation*}
            \partial_tf \leq \dive\left (B\nabla f + b \right),
        \end{equation*}
        with $B$ and $b$ as in the previous theorem, there exists a constant $C>0$ such that
	    \begin{equation}\label{eqn: Stampacchia estimate}
	      \|f\|_{L^\infty(Q)}\leq C\left (\|f\|_{L^\infty(\partial_p Q)}+ \|b\|_{L^d(Q)}	\right),
        \end{equation}
        as before, $C$ is determined by $\lambda,\Lambda,d$ and $R$.
    \end{thm}

    We also recall the interior classical regularity estimates when the coefficients are H\"older continuous in time and space. See \cite[Chapter IV]{LadSolUra1968} or also \cite[Chapter III, p. 59]{Lie1996} for a proof.
    \begin{thm}\label{thm: Schauder} (Schauder estimates.) If $B,b \in C^{\beta;\beta/2}(Q)$, then there is a finite $C$ such that
	                \begin{equation*}
		                	\;[D^2f]_{C^{\beta,\alpha/2}(Q_{1/2})}+[\partial_t f]_{C^{\beta,\beta/2}(Q_{1/2})} \leq C\left ( \lambda,\Lambda,R, \|B\|_{C^{\beta;\beta/2}(Q)},\|b\|_{C^{\beta;\beta/2}(Q)},\|f\|_{L^\infty(Q)} \right ). 
	                \end{equation*}	 	

    \end{thm}

   %%%%%%%%%%%%%%%%%%%%%%%%%%%%%%
   %%%%%%%%%%%%%%%%%%%%%%%%%%%%%%
   %%%%%%%%%%%%%%%%%%%%%%%%%%%%%
   
  \section{Radial symmetry}\label{sec: radial symmetry}
  This section is devoted to some technical lemmas. The proofs of the first two propositions are rather technical and can be found in the Appendix. 

  \begin{prop}\label{prop: preservation of radiality}
      Suppose $\fin$ and $g(\cdot,t)$ are both radially symmetric, and let $Q(\cdot,\cdot)$ denote either $\QL$ or $\QKS$. Then any solution of the linear Cauchy problem 
	  \begin{equation*}
    	 \partial_t f = Q(g,f), \quad f(v,0)=f_{in}(v),
      \end{equation*}
	  is radially symmetric for all $t$. Furthermore, if $\fin$ and $g$ are radially decreasing, then so is $f$.
    \end{prop}
 
    Let $h:\mathbb{R}^3\to\mathbb{R}_+$, define
    \begin{equation}\label{eqn: A eigenvalue radial direction}
    	 A^*[h](v) := (A[h](v)\hat v,\hat v),\;\;v\neq 0,\quad \hat v := v|v|^{-1}.
    \end{equation}
    There are two useful expressions for $A^*[h]$ and $a[h]$ when $h$ is radially symmetric.
    \begin{prop}\label{prop: Landau in radial coordinates}
    	 Let $h \in L^1(\mathbb{R}^3)$ be radially symmetric and non-negative. Then
         \begin{align}\label{A_symm}
         	  A^*[h](v) & = \frac{1}{12 \pi |v|^3}\int_{B_{|v|}}h(w)|w|^2\;dw+\frac{1}{12 \pi}\int_{B_{|v|}^c}\frac{h(w)}{|w|}\;dw,\\
              a[h](v) & = \frac{1}{4\pi |v|}\int_{B_{|v|}}h(w)\;dw+\frac{1}{4\pi} \int_{B_{|v|}^c}\frac{h(w)}{|w|}\;dw.\label{a_symm}
         \end{align}
    \end{prop}
	The second formula above is simply the classical formula for the Newtonian potential in the case of radial symmetry, the formula for $A^*[h]$ is new and 
        
    \begin{lem}\label{lemma: A^* bounds}
      Let $h \in L^1(\R^3)$ be a non-negative, decreasing radial function.\\
      
      (1) If \begin{equation*}
        \int_{B_{R_1}\setminus B_{R_0}} h \;dv \geq \delta>0,
      \end{equation*}
      for some $\delta>0$ and $0<R_0<R_1$ then,
      \begin{align}\label{bound_below_A[g]}
        A[h](v)\geq \frac{\delta R_0^2}{12\pi(1+R_1^3)}\frac{1}{1+|v|^3} \Id.
      \end{align}

      (2) If $h$ is bounded, i.e. if $\|h\|_{L^\infty(\mathbb{R}^3)} = h(0) <+\infty$, it holds
      \begin{align}\label{bound_above_A[g]}
        A[h](v)\leq a[h]\Id  \le 2\left (\frac{ \|h\|_{L^\infty(\mathbb{R}^3)}+\|h\|_{L^1(\mathbb{R}^3)} }{1+|v|}\right )\;\Id.
      \end{align}
    \end{lem}

    \begin{proof}    
    (1)   Let $A^*[h]$ be as in (\ref{A_symm}). If $|v|\geq R_1$, then 
     \begin{align*}
         A^*[h](v) &\ge \frac{1}{12\pi |v|^3} \int_{B_{R_1}} {h(w)}{|w|^2} \;dw  \ge  \frac{1}{12\pi |v|^3} \int_{B_{R_1}\setminus B_{R_0}} {h(w)}{|w|^2} \;dw\\
         &\geq  \frac{R_0^2}{12\pi |v|^3}    \int_{B_{R_1}\setminus B_{R_0}} {h(w,t)} \;dw \ge     \frac{\delta R_0^2}{12\pi |v|^3}.
     \end{align*}
     Note that Proposition \eqref{prop: Landau in radial coordinates} guarantees that $A^*[h]$ is radially decreasing. Thus,
     \begin{equation*}
         A^*[h](v) \geq \frac{\delta R_0^2}{12\pi R_1^3},\;\;\forall\;v \in B_{R_1}.
     \end{equation*}
     Combining both estimates, we conclude that
     \begin{equation*}
         A^*[h](v) \geq \frac{\delta R_0^2}{12\pi(1+R_1^3)}\frac{1}{1+|v|^3}. 
     \end{equation*}
     
     (2) If $h \in L^\infty$, we may use \eqref{A_symm} to obtain the estimate
     \begin{eqnarray*}
         A[h]\leq a[h](v)\Id &\leq &  \left(\frac{h(0)}{4\pi |v|} \int_{B_{|v|}}\;dw+\frac{1}{4\pi}\int_{B_1^c}h(w)\;dw + \frac{1}{4\pi}\int_{B_1}\frac{h(w)}{|w|}\;dw \right) \Id \\
       &  \le &\left(\|h\|_{L^\infty(\R^3)} +  \|h\|_{L^1(\R^3)} \right)\Id, \quad \text{ if } |v|\leq1,       
     \end{eqnarray*}
     and 
     \begin{eqnarray*}
        A[h]\leq a[h](v) \Id \leq &\left(\frac{\|h\|_{L^1(\mathbb{R})}}{2 \pi |v|} \right)\Id,\quad \text{ if } |v|\geq 1. \\
       % \leq & \left(\frac{ \|h\|_{L^\infty(\mathbb{R}^3)}+\|h\|_{L^1(\mathbb{R}^3)} }{4\pi|v|}\right )\Id, \text{ if } |v|\geq 1.         
     \end{eqnarray*}

    \end{proof}
    
     \begin{prop}\label{lem: classical radial barriers}
    	 Let $h$ be a positive and radially symmetric decreasing function. For any $\gamma \in (0,1)$ define $U_\gamma(v)$ as 
	     \begin{align*}%\label{new_M}
	       U_\gamma(v):=|v|^{-\gamma}.
	     \end{align*}
	     Then,
         \begin{equation*}
         	  \QL(h,U_\gamma),\;  \QKS(h,U_\gamma)\leq U_\gamma \left ( -\tfrac{1}{3}\gamma(1-\gamma) a[h]|v|^{-2}+h\right ).
         \end{equation*} 
        % and
       %  \begin{equation*}
       %  	  \QKS(h,U_\gamma)\leq U_\gamma \left ( -\tfrac{1}{3}\gamma(1-\gamma) a[h]|v|^{-2}+h \right ).
       %  \end{equation*} 
    \end{prop}
    \begin{proof}
	     As $U_\gamma$ is radial
	     \begin{equation*}
	     	  \nabla U_\gamma(v) = U'_\gamma(v)\tfrac{v}{|v|},\;\;D^2 U_\gamma(v) = U''_\gamma(v)\tfrac{v}{|v|}\otimes \tfrac{v}{|v|}+U'_\gamma(v)\tfrac{1}{|v|}(\Id - \tfrac{v}{|v|}\otimes \tfrac{v}{|v|}).
	     \end{equation*}
	     Thus, 
    	 \begin{align*}
    	 	  \QL(h,U_\gamma) & = \tr (A[h]D^2U_\gamma)+hU_\gamma= A^*[h] U_\gamma''+\frac{a[h]-A^*[h]}{|v|}U_\gamma'+h U_\gamma.	
    	 \end{align*}
         In particular, since $U_\gamma'= -\gamma r^{-1}U_\gamma$, $U_\gamma''= \gamma(\gamma+1)|v|^{-2}U_\gamma$, it holds
         \begin{equation*}
         	  \QL(h,U_\gamma) = U_\gamma \left ( \gamma(\gamma+1)A^*[h]|v|^{-2}-\gamma (a[h]-A^*[h])|v|^{-2}+h \right ).
         \end{equation*}
         The thesis follows by noticing that $ A^*[h] \le  \frac{1}{3} a[h]$. \\          
                  As for $\QKS(h,U_\gamma)$, straightforward computations show that 
         $$
         \QKS(h,U_\gamma) = U_\gamma\left( -\gamma(1-\gamma)a[h]|v|^{-2} +h\right).
         $$ 
    \end{proof}

   \section{Short time existence.}\label{short_time_existence}

  In the following section, the operator $Q$ denotes either $\QL$ or $\QKS$, unless otherwise specified. A sequence $\{ f_k^{(\delta)} \}_k$ of approximate solutions to the non-linear Cauchy problem for $Q$ is built iteratively as follows. First set $f^{(\delta)}_0(v,t) :=  \fin(v)$ for all $v$ and $t>0$. Then, for {{$\fin$ as in (\ref{init_data_c_beta}) and \eqref{eqn: initial assumptions}}}, $k\in \mathbb{N}$ let $f^{(\delta)}_{k}$ be the unique classical solution of
	    \begin{equation} \label{eqn: approximating problems}
	     \left \{ \begin{array}{rl}
	          \partial_t f & = \delta \Delta f+Q(f_{k-1}^{(\delta)},f) \text{ in } \mathbb{R}^3 \times \mathbb{R}_+,\\
	          f(\cdot,0) & =  \fin, %\quad \fin \;\textrm{as in (\ref{init_data_c_beta}) and \eqref{eqn: initial assumptions}} 
	     \end{array}\right.	
	    \end{equation}
   which exists thanks to Lemma \ref{lem: existence uniqueness regularity for linear problem}. Moreover, each $f_k^{(\delta)}$ is radially symmetric and monotone, by Proposition \ref{prop: preservation of radiality} and the assumptions \eqref{eqn: initial assumptions} for $\fin$. 
   
   By virtue of $f_k^{(\delta)}$ being smooth and monotone decaying, it follows that $f_k^{(\delta)}(0,t)\geq f_k^{(\delta)}(v,t)$ and $D^2f^{(\delta)}(0,t)\leq 0$. Thus, $f_k^{(\delta)}(0,t)$, seen solely as a function of $t$, always satisfies the differential inequality
	\begin{align}\label{eqn: quadratic ordinary differential inequality}
	 \partial_t f_k^{(\delta)}(0,t)\leq f_{k-1}^{(\delta)}(0,t) f_k^{(\delta)}(0,t),\;\; k\ge1.
	\end{align}
	We next prove that $\|f_k^{(\delta)}(\cdot,t)\|_{L^\infty(\R^3)}$ is uniformly bounded in $\delta,k$ for $t\in[0,T_0]$, where $T_0$ depends only on $f_{in}$. 
	\begin{lem}\label{lem:short time}   	   
	Let $f_k^{(\delta)}(v,t)$ be given by \eqref{eqn: approximating problems}. Then,
	\begin{equation}\label{y_k}
	   f_k^{(\delta)}(0,t)\leq \frac{f_{in}(0)}{1-f_{in}(0)t}, \quad  \forall\; t\in\left[0, \frac {1}{f_{in}(0)}\right).
	\end{equation}
			    
	\end{lem}
	    
	\begin{proof} 
    Define functions $h_k(t)$ iteratively for $k \in \mathbb{N}$ by 
	\begin{align*}%\label{eqn: quadratic ordinary differential inequality}
	h'_k(t) = h_{k-1}(t)h_k(t), \quad h_k(0)= \fin(0),\quad  h_0 \equiv \fin(0).
	\end{align*}
    First we observe that $h_{k}\leq h_{k+1}$ for all $k$; hence $h'_k \le h_{k}^2$ for all $k\ge0$. The thesis follows since 
   \begin{equation*}
      f_k^{(\delta)}(0,t)   \le h_k(t)\le  \frac{f_{in}(0)}{1- f_{in}(0)t}.
    \end{equation*}
    
	\end{proof}
	    
    \begin{prop}\label{prop: mass lower bound short time}
         Let $\{f_k\}$ be given by \eqref{eqn: approximating problems} and suppose $T>0$ is such that
		 \begin{equation*}
		   C(T) := \limsup \limits_k \|f_k^{(\delta)}\|_{L^\infty(\mathbb{R}^3\times [0,T])} <\infty 
		 \end{equation*}
		 Then, there are $r=r(T,C(T),\fin)$ and $R=R(T,C(T),\fin)$, with $r<R$ and such that
         \begin{equation}\label{eqn: lower bound on a_proof}
         	  \int_{B_R \setminus B_r} f_k^{(\delta)}(v,t)\;dv \geq \frac{1}{2}\;\;\forall\; k\ge 1 \quad t\in[0,T) .
         \end{equation}
    \end{prop}

  \begin{proof}
    First, let us compute the rate of change for the second moment of $f$, using the evolution equation for $f$:
	\begin{equation*}
      \frac{d}{dt}\int_{\mathbb{R}^3}f_k^{(\delta)}(v,t)|v|^2\;dv = \int_{\mathbb{R}^3}\left ( \delta \Delta f_k^{(\delta)}+Q(f_{k-1}^{(\delta)},f_k^{(\delta)})\right )|v|^2\;dv.
	\end{equation*}
    For simplicity, $B[g]$ will to denote either the matrix $A[g]$ for $ Q= \QL$ or the matrix $a[g]\Id$ for  $Q= \QKS$,. Integration by parts above yields
	\begin{align*}
      \frac{d}{dt}\int_{\mathbb{R}^3}f_k^{(\delta)}(v,t)|v|^2\;dv & = -2\int_{\mathbb{R}^3}(B[f_{k-1}^{(\delta)}]\nabla f_k^{(\delta)}-f_k^{(\delta)} \nabla a[f_{k-1}^{(\delta)}],v)\;dv\\
        & = -2\int_{\mathbb{R}^3}(\nabla f_k^{(\delta)},B[f_{k-1}^{(\delta)}]v)\;dv+2\int_{\mathbb{R}^3}f_k^{(\delta)} (\nabla a[f_{k-1}^{(\delta)}],v)\;dv\\
        & = 2\int_{\mathbb{R}^3}f_k^{(\delta)} (\dive(B[f_{k-1}^{(\delta)}]v)+(\nabla a[f_{k-1}^{(\delta)}],v))\;dv	.	  
	\end{align*}
    Note that $\dive(B[f_{k-1}^{(\delta)}]v)+(\nabla a[f_{k-1^{(\delta)}}],v)=\tr[B[f_{k-1}^{(\delta)}]\Id]+2(\nabla a[f_{k-1}^{(\delta)}],v)$, whether $B[\cdot]$ is given by $A[\cdot]$ or $a[\cdot]$. Moreover,
	\begin{equation*}
      \tr(B[f_{k-1}^{(\delta)}])\leq 3a[f_{k-1}^{(\delta)}],
	\end{equation*}
	and
	\begin{equation*}
      \nabla a[f_{k-1}^{(\delta)}](v,t) = -\frac{1}{4\pi}\int_{\mathbb{R}^3}f_{k-1}^{(\delta)}(w,t)\frac{(v-w)}{|v-w|^3}\;dw.		
    \end{equation*}		
	Hence 
	
	\begin{align*}
      |\nabla a[f_{k-1}^{(\delta)}](v,t)| & \leq \frac{1}{4\pi}\int_{\mathbb{R}^3}\frac{f_{k-1}^{(\delta)}(w,t)}{|v-w|^2}\;dw\\
	    & = \frac{1}{4\pi}\int_{B_1(v)}\frac{f_{k-1}^{(\delta)}(w,t)}{|v-w|^2}\;dw+\frac{1}{4\pi}\int_{B_1(v)^c}\frac{f_{k-1}^{(\delta)}(w)}{|v-w|^2}\;dw\\
		& \leq \frac{1}{4\pi}4\pi\|f_{k-1}^{(\delta)}(\cdot,t)\|_{L^\infty(\mathbb{R}^3)}+\frac{1}{4\pi}\|f_{k-1}^{(\delta)}(\cdot,t)\|_{L^1(\mathbb{R}^3)}.
	\end{align*}	
	Therefore, for $t<T$
	\begin{eqnarray*}
      \frac{d}{dt}\int_{\mathbb{R}^3}f_k^{(\delta)}|v|^2\;dv \leq & 16 \|f_{k-1}^{(\delta)}(\cdot,t)\|_{L^1(\mathbb{R}^3)} \left( \|f_{k-1}^{(\delta)}(\cdot,t)\|_{L^\infty(\mathbb{R}^3)}+\|f_{k-1}^{(\delta)}(\cdot,t)\|_{L^1(\mathbb{R}^3)}\right) \\
     & + 4  \left( \|f_{k-1}^{(\delta)}(\cdot,t)\|_{L^\infty(\mathbb{R}^3)}+\|f_{k-1}^{(\delta)}(\cdot,t)\|_{L^1(\mathbb{R}^3)}\right) \int_{\R^3} f_k^{(\delta)} |v|^2 \;dv.
      \end{eqnarray*}
      which implies, %since $\| f_{in}\|_{L^1(\R^3)} =1$
  %    (C(T)+\|\fin\|_{L^1(\mathbb{R}^3)})\int_{\mathbb{R}^3}f_k^{(\delta)}(v,t)\;dv = (C(T)+\|\fin\|_{L^1(\mathbb{R}^3)})\|\fin\|_{L^1(\mathbb{R}^3)}		
    %\end{equation*}	
	%Since $\|\fin\|_{L^1(\mathbb{R}^3)}=1$
	\begin{equation}
	  \int_{\mathbb{R}^3}f_{k}^{(\delta)}(v,t)|v|^2\;dv \leq e^{  4T( C(T)+M_{in})}
	  \left( \int_{\mathbb{R}^3}f_{in}|v|^2\;dv + 4\right).\label{II_mom_bound}
	 % \int_{\mathbb{R}^3}\fin|v|^2\;dv+ T (C(T)+1),\;\;\forall\;t\in (0,T).
    \end{equation}
	On the other hand, for any $r,R$ with $R>r>0$ there is the obvious lower bound,	
    \begin{align*}		
      \int_{B_R\setminus B_r} f_k^{(\delta)}(v,t)\;dv & = 1-\int_{B_R^c}f_k^{(\delta)}(v,t)\;dv-\int_{B_r}f_k^{(\delta)}(v,t)\;dv\\
	    & \geq 1 - R^{-2}\int_{\mathbb{R}^3}f_k^{(\delta)}(v,t)|v|^2\;dv-\frac{4}{3} \pi r^3 C(T).
    \end{align*}
	The thesis follows by choosing 
	\begin{align*}
      R & := \sqrt{6}\;\left ( \int_{\mathbb{R}^3}\fin|v|^2\;dv+ 4\right )^{1/2} e^{2T(C(T)+M_{in})},\\
	  r & := (4\pi C(T))^{-1/3}, 
    \end{align*}
    since
	\begin{equation*}
	  1 - R^{-2}\int_{\mathbb{R}^3}f_k^{(\delta)}(v,t)|v|^2\;dv-\frac{4}{3} \pi r^3 C(T)\geq \frac{1}{2}.
    \end{equation*}
  \end{proof}

     \begin{prop}\label{prop: A is C^1,alpha and strictly positive}
		     Let $f_k^\delta$  be the unique solution to \eqref{eqn: approximating problems}, and $M_{in}:=\int_{\R^3} f_{in}(v)\;dv$. Then
		     \begin{enumerate}
		     \item There exists a constant $c_0=c_0(\fin,T)$ such that 
    \begin{align}\label{bound_below_A[f]}
    A^*[f_k^\delta](v)\geq \frac{c_0}{1+|v|^3}\Id, \quad \forall \; t \in [0,T],\quad \forall k\ge1.
    \end{align} 
   
         \item $A[f_k]$ and $a[f_k]$ satisfy the pointwise bound
				       \begin{equation}\label{eqn: upper bound for A}
				               A[f_k^\delta](v,t) \leq a[f_k^\delta](v,t)\le \frac{2(\|f_k^\delta\|_{L^\infty(\R^3 \times [0,T])} + M_{in})}{1+|v|}\Id, \quad \forall\; v \in \mathbb{R}^3.
				       \end{equation}
				        
				        %with $C_0$ as in \eqref{Linfty-bound_alpha}.
				  \item $A[f_k^\delta](\cdot,t)\in C^{1,\alpha}(\mathbb{R}^3)$ for every $\alpha \in [0,1]$ with a bound of the form,
				       \begin{equation*}
				           \| \nabla A[f_k^\delta]\|_{C^\alpha({\mathbb{R}^3})}\leq c(M_{in}, \|f_k^\delta\|_{L^\infty(\R^3\times [0,T])}).	
				       \end{equation*}
				  \item $A[f_k^\delta]$ is locally H\"older continuous in time, and
				       \begin{equation*}
				           \| A[f_k^\delta](v,\cdot)\|_{C^\beta[0,T]}\leq c(\fin,R,\|f_k^\delta\|_{L^\infty(\R^3\times [0,T])}),\;\;\forall\;v\in B_R.
				       \end{equation*}
				  					   
		     \end{enumerate}
    \end{prop}

	\begin{proof}
	The proof of (1) and (2) follows directly from Proposition \ref{prop: mass lower bound short time} and Lemma \ref{lemma: A^* bounds}. 
	For (3) first notice that $a[f_k^\delta]\in W^{2,p}(\R^3)$ for any $1\le p\le +\infty$ since $f_k^\delta\in L^1(\mathbb{R}^3)\cap L^\infty(\mathbb{R}^3)$.
          In particular 
		     $$
		     \|a[f_k^\delta]\|_{W^{2,p}}\le C\|f_k^\delta\|_{L^p} \le C\|f_k^\delta\|_{L^\infty(\R^3)}^{p-1}\|f_k^\delta\|_{L^1(\R^3)},
		     $$
		     where $C$ is a dimensional constant. 
		     Hence Morrey's inequality implies $a[f_k^\delta]\in C^{1,\alpha}(\R^3)$ for any $0<\alpha<1$. 
		     We use now the fact that $A[f_k^\delta]$ can be rewritten as 
		     $$
		     A[f_k^\delta]=D^2(\Delta^{-2}f_k^\delta),
		     $$
		     with $D^2$ being the Hessian matrix, an apply $W^{2,p}$ Calderon-Zygmund estimates to the kernels $\partial_{i,j}(\Delta^{-2}f)$ we get
		     we get:
		     $$
		     \|D^2 A[f_k^\delta]\|_{L^p(\R^3)} \le A_p \|a[f_k^\delta]\|_{W^{2,p}},
		     $$
		     with $A_p$ a dimensional constant.  \\
		     	  To prove (4) consider any function $h:\mathbb{R}^3\to\mathbb{R}$: it holds
	  %$$
%	%   |A[h](v_0)| \le  2\left (\frac{ \|h\|_{L^\infty(\mathbb{R}^3)}+\|h\|_{L^1(\mathbb{R}^3)} }{1+|v_0|}\right )\;\Id.
%$$

	  \begin{align*}
	    |A[h](v_0)|\le |a[h](v_0)| & \leq \frac{1}{8\pi}\int_{B_R(v_0)}\frac{h(v)}{|v_0-v|}dv +\frac{1}{8\pi}\int_{B_R(v_0)^c}\frac{h(v)}{|v_0-v|}dv\\
		            & \leq \frac{1}{8\pi}\|h\|_{L^\infty(B_R(v_0))}\frac{4\pi}{2}R^2+\frac{1}{8\pi R} \|h\|_{L^1(B_R(v_0)^c)}\\
		            & \leq \frac{R^2}{4}\|h\|_{L^\infty(B_R(v_0))}+\frac{1}{8\pi R}\|h\|_{L^1(B_R^c(v_0))}.					
	  \end{align*}
	  The above estimate for $h=f_k^\delta(\cdot,t)-f_k^\delta(\cdot,s)$ ($t,s\in [0,T]$) yields
	  \begin{equation*}
        |A[f_k^\delta](v_0,t)-A[f_k^\delta](v_0,s)| \leq \frac{R^2}{4}\|f_k^\delta(\cdot,t)-f_k^\delta(\cdot,s)\|_{L^\infty(B_R(v_0))}+\frac{1}{8\pi R}\|f_k^\delta(\cdot,t)-f_k^\delta(\cdot,s)\|_{L^1(B_R^c(v_0))}.
      \end{equation*}
	  Since $f_k^\delta$ is locally H\"older continuous in time, (4) follows.	  
	\end{proof}	
	
	\begin{prop}\label{Prop:cont}
	  Let $w_k := f_{k}^{(\delta)}-f_{k-1}^{(\delta)}$ and
	  \begin{equation*}
        T_1 := \sup \{ t\geq 0 \mid \limsup \limits_{k} \|f^{(\delta)}_k\|_{L^\infty(\mathbb{R}^3\times [0,t])}<\infty\}.	   
	  \end{equation*} 
	  If $0<T<T_1$, then 
	  \begin{align}\label{cont}
        \|w_{k}\|_{L^\infty( [0,T], L^\infty \cap L^1 (\R^3))}
        	& \leq \frac{1}{2}\ \|w_{k-1}\|_{L^\infty( [0,T], L^\infty \cap L^1 (\R^3))}.  	  
      \end{align}	
      	
	 % \begin{equation*}
    %    \lim \limits_{k\to\infty} \|w_k\|_{L^\infty(\mathbb{R}^3\times [0,T])} = \lim \limits_{k\to \infty }\sup \limits_{t\in [0,T]} \|w_k(\cdot,t)\|_{L^1(\mathbb{R}^3)}= 0.
	%  \end{equation*}
	\end{prop}
	
	\begin{proof}
      \emph{Step 1.} The function $w_k$ solves	
      \begin{equation*}
        \left \{ \begin{array}{rll}
			\partial_t w_k & = &  \delta \Delta w_k +Q(f_{k-1}^{(\delta)},w_k)+F,\\
			w_k(\cdot,0) & = & 0,
        \end{array}	\right.		
      \end{equation*}
	  where
	  \begin{equation*}
	    F(w_{k-1},f_{k-1}^{(\delta)}) := Q(w_{k-1},f_{k-1}^{(\delta)}).
      \end{equation*}		  
      The bound of the second moment (\ref{II_mom_bound}) for $f_{k}^{(\delta)}$ implies that 
      \begin{equation*}%\label{hessian}
        \|f_k^{(\delta)}\|_{L^\infty(B_1(v)\times [0,T])}\leq C(T)(1+|v|^5)^{-1},\;\;\forall\;v.		  
	  \end{equation*}
      Theorems \ref{thm: parabolic regularity} and \ref{thm: Schauder} imply that there exists a positive constant $\tilde{C}= \tilde{C}(M_{in},C(T))$ such that 
      
	  %Since $\limsup \limits_k \|f_k^{(\delta)}\|_{L^\infty(\mathbb{R}^3\times [0,T])} \leq C(T)$, it holds:
	  \begin{equation}\label{hessian}
        \|D^2f_k^{(\delta)}\|_{L^\infty(B_1(v)\times [0,T])}\leq \tilde{C}(1+|v|^5)^{-1},\;\;\forall\;v.		  
	  \end{equation}
	  Moreover from
	  \begin{align}\label{Aa}
	    \|A[h(\cdot,t)]\|_{L^\infty([0,T]\times \mathbb{R}^3)} &\leq \|a[h]\|_{L^\infty([0,T]\times \mathbb{R}^3)}\Id \\
	    &\le ( \|h(\cdot,t)\|_{L^\infty([0,T]\times \mathbb{R}^3)}+\|h(\cdot,t)\|_{L^1([0,T]\times \mathbb{R}^3)})\Id, \nonumber
	  \end{align}
	  it follows that, for $v\in\mathbb{R}^3$, $t\in (0,T_0)$,
	  \begin{equation*}
	    |F(v,t)|\leq \tilde{C}\left ( \frac{\|w_{k-1}(\cdot,t)\|_{L^\infty([0,T]\times \mathbb{R}^3)}+\|w_{k-1}(\cdot,t)\|_{L^1([0,T]\times \mathbb{R}^3)}}{(1+|v|^5)} \right ).
	  \end{equation*}	  
	  %Since
	  %\begin{equation*}
      %  \|a[w_{k-1}]\|_{L^\infty}\leq \|w_{k-1}\|_{L^\infty}+\|w_{k-1}\|_{L^1}
     % \end{equation*}
	  
      \emph{Step 2.} Let $h(v) := (1+|v|^2)^{-2}$, then
	  \begin{equation*}
        D^2h = -4(1+|v|^2)^{-3}\Id+{{24}}(1+|v|^2)^{-4}v\otimes v.
	  \end{equation*}
	  So,
	 {{ \begin{align*}
        |\tr( (\delta \Id +A[f_{k-1}^{(\delta)}])D^2h) | \leq &\frac{12\delta}{(1+|v|^2)^{3}}\left ( -1+ \frac{2|v|^2}{(1+|v|^2)}\right)\\
        & -4 \frac{a[f_{k-1}^{(\delta)}]}{  (1+|v|^2)^{3}} +   \frac{24}{  (1+|v|^2)^{4}} (A[f_{k-1}^{(\delta)}]v,v)\\
        \le & \frac{12\delta}{(1+|v|^2)^{3}}  -4 \frac{a[f_{k-1}^{(\delta)}]}{  (1+|v|^2)^{3}} + 8 \frac{a[f_{k-1}^{(\delta)}]|v|^2}{  (1+|v|^2)^{4}} \\
        \le & \frac{12\delta}{(1+|v|^2)^{3}}  +4 \frac{a[f_{k-1}^{(\delta)}]}{  (1+|v|^2)^{3}}  \left[ -1 + \frac{2|v|^2}{(1+|v|^2)}\right]\\
        \le & 12\delta h +4 \frac{a[f_{k-1}^{(\delta)}]}{  (1+|v|^2)^{3}} \le 12\delta h +4 h (C(t)+1),
      \end{align*}}}
      taking into account (\ref{Aa}) and the fact that $(A[f_{k-1}^{(\delta)}]v,v)=A^*[f_{k-1}^{(\delta)}](v)|v|^2$, $\tr A^*[f_{k-1}^{(\delta)}] = a[f_{k-1}^{(\delta)}]$ and $A^*[f_{k-1}^{(\delta)}](v)\le \frac{1}{3} a[f_{k-1}^{(\delta)}]$. Hence for $\delta <1$ and $C(T)>1$ it holds
	  \begin{align*}
         |\delta \Delta h + Q(f_{k-1}^{(\delta)},h)| & = |\tr( (\delta \Id +A[f_{k-1}^{(\delta)}])D^2h)+f_{k-1}^{(\delta)}h| \\
		   & \leq  12\delta h +4 h +5hC(t)\\
		   & \leq 6C(T)h.
      \end{align*}	  
	  Next, let $H(v,t):= A^{-1}(e^{At}-1)Bh(v)$, for $A,B>0$ to be determined; a straightforward computation shows that
	  \begin{equation*}
        \partial_t H = AH+Bh.		  
	  \end{equation*}
	  Hence
	  \begin{align*}
        \delta \Delta H+Q(f^{(\delta)}_{k-1},H)+F \leq &\; 6 C(T)H\\
        &+\tilde{C}h\left ( {\|w_{k-1}(\cdot,t)\|_{L^\infty([0,T]\times \mathbb{R}^3)}+\|w_{k-1}(\cdot,t)\|_{L^1([0,T]\times \mathbb{R}^3)}} \right ).
	  \end{align*}
	  By taking
	  \begin{align*}
	   A:=6C(T), \quad B := 2\tilde{C}( {\|w_{k-1}(\cdot,t)\|_{L^\infty([0,T]\times \mathbb{R}^3)}+\|w_{k-1}(\cdot,t)\|_{L^1([0,T]\times \mathbb{R}^3)}}),
	   \end{align*}
	    it follows that
	  \begin{equation*}
	    \partial_t H \geq \delta \Delta H+Q(f^{(\delta)}_{k-1},H)+F.
	  \end{equation*}
	  This means that $H$ is a supersolution for the same parabolic equation solved by $w_{k}$. Moreover, $H(\cdot,0)=w_{k}(\cdot,0)=0$. Then, the comparison principle implies that
	  \begin{equation*}
        w_k\leq H \text{ in } \mathbb{R}^3\times \mathbb{R}_+.
	  \end{equation*} 
	  If $t\in(0,1)$, then for all $1\le p \le +\infty$ it holds
	  \begin{align*} 
       \|w_k(\cdot,t)\|_{L^p(\mathbb{R}^3)} \le  \|H(\cdot,t)\|_{L^p(\mathbb{R}^3)} = BA^{-1}(e^{At}-1)\|h\|_{L^p(\mathbb{R}^3)} \leq 2e^ABt .%\|h\|_{L^p(\mathbb{R}^3)}.
	  \end{align*} 
	  Hence for $T_0<1$ such that
	  \begin{align*}
	  T_0\le \frac{1}{16e^{6C(T)}\tilde{C}},
	  \end{align*}
	  we have the following inequality:
	  \begin{align*}
        \|w_{k}\|_{L^\infty(\mathbb{R}^3\times [0,T_0])}+\sup \limits_{t\in [0,T_0]} \|w_{k}(\cdot,t)\|_{L^1(\mathbb{R}^3)} 
        %& \leq \|H\|_{L^\infty(\mathbb{R}^3\times [0,T_0])}+\sup \limits_{t\in [0,T_0]} \|H(\cdot,t)\|_{L^1(\mathbb{R}^3)}\\
        	& \leq \frac{1}{2}\left ( \|w_{k-1}\|_{L^\infty(\mathbb{R}^3\times [0,T_0])}+\sup \limits_{t\in [0,T_0]} \|w_{k-1}(\cdot,t)\|_{L^1(\mathbb{R}^3)}\right ).  	  
      \end{align*}		  
      At time $T_0$ we shell restart the process that will lead to the proof of estimate (\ref{cont}) in the time interval $[T_0,2T_0]$.  Hence we can iterate the process and cover the whole time interval $[0,T]$. 

	\end{proof}

    \begin{lem}\label{lem: short time existence regularized}
	  Let $\{f_k^{(\delta)}\}$ be the sequence defined via \eqref{eqn: approximating problems} with either $Q=\QL$ or $Q=\QKS$.  Let 
	  \begin{equation*}
        T_1 := \sup \{ t\geq 0 \mid \limsup \limits_{k} \|f^{(\delta)}_k\|_{L^\infty(\mathbb{R}^3\times [0,t])}<\infty\}.	   
	  \end{equation*} 
	  For $0<T<T_1$ the sequence of functions $f_k^{(\delta)} $ converges in $C^{2 ; 1}_{\textnormal{loc}}\left ( \mathbb{R}^3\times [0,T_0)\right )$ to a function $f^{(\delta)} $ that solves 
	   \begin{equation*} %\label{eqn: approximating problems}
	     \left \{ \begin{array}{rl}
	          \partial_t f^{(\delta)} & = \delta \Delta f^{(\delta)}+Q(f^{(\delta)},f^{(\delta)}) \text{ in } \mathbb{R}^3 \times \mathbb{R}_+,\\
	          f(\cdot,0) & =  \fin.
	     \end{array}\right.	
	    \end{equation*}	%  , there exists a strictly positive and possibly infinite $T_0$, such that
   %   \begin{equation*}
   %     \textnormal{(along a subsequence,  } \delta \textnormal{ fixed)}\; f_k^{(\delta)} \to f^{(\delta)} \textnormal{ in } C^{2 ; 1}_{\textnormal{loc}}\left ( \mathbb{R}^3\times [0,T_0)\right ),
 %     \end{equation*}
   %   where $f^{(\delta)}(v,0)=\fin(v)$ and $f$ solves \eqref{eqn: Landau} or \eqref{eqn: Krieger Strain}, depending on $Q$.  
   Moreover, $T_1= \infty$ or 
      \begin{equation*}
        \lim \limits_{t \to T_1^-} \|f^{(\delta)}(\cdot,t)\|_{L^\infty(\mathbb{R}^3)} = \infty.
      \end{equation*}
      In either case we have the estimate $\frac{1}{2\fin(0)}<T_1$.
	\end{lem}
	
	\begin{proof}	
	Let $T<T_1$. Proposition \ref{Prop:cont} implies that the operator $S: f^\delta_{k-1} \to f^\delta_{k}$ has a fixed point $f^\delta$ in the space ${L^\infty( [0,T], L^\infty \cap L^1 (\R^3))}$, such that 
	\begin{equation*}
		    \|f^{(\delta)}\|_{L^\infty(\mathbb{R}^3\times [0,T])}  \le C(T). 
		 \end{equation*}
		 	%  Consider
  %    \begin{equation*}
   %     T_0 := \sup \{ t\geq 0 \mid \limsup\limits_{k} \|f_k^{(\delta)}\|_{L^\infty(\mathbb{R}^3\times [0,t])} < \infty \},
   %   \end{equation*}
   %   Clearly $T_0> \frac{1}{2\fin(0)}$ and possibly $T_0=+\infty$. If $T_0 <\infty$, then
  %    \begin{equation*}
  %        \left \{ \begin{array}{rl}          
  %%         \limsup \limits_{k}\; \|f_{k}^{(\delta)}\|_{L^\infty(\mathbb{R}^3\times [0,t])} & = +\infty,\;\;\forall\; t>T_0,\\
   %        \limsup\limits_{k} \|f_{k}^{(\delta)}\|_{L^\infty(\mathbb{R}^3\times [0,t])} & < +\infty,\;\;\forall\; t<T_0.
   %       \end{array}\right.
   %   \end{equation*}      
      In particular, the functions $f_k^{(\delta)}$ stay uniformly bounded in compact subsets of $\mathbb{R}^3\times [0,T_1)$. By the regularity estimates in Theorem \ref{thm: parabolic regularity}, it follows that $D^2f_k^{(\delta)}$ and $\partial_t f_k^{(\delta)}$ are also uniformly in $k$ and $\delta$ H\"older continuous in compact subsets of $\mathbb{R}^3\times [0,T_1)$. 
      
      %A standard Cantor diagonalization argument guarantees the existence of a subsequence, $f^{(\delta)}_{k_j}$, such that 
      Hence $f_{k}^{(\delta)}$ converges uniformly in compact subsets of $\mathbb{R}^3\times [0,T_1)$ (together with $\partial_t f_{k}^{(\delta)}$, $\nabla f_{k}^{(\delta)}$, $D^2f_{k}^{(\delta)}$) to some function $f^{(\delta)}$ (and its corresponding derivatives).		% Moreover, according to Proposition ??,
	 % \begin{equation*}
	 %   \lim \limits_{j} \|w_{k_j}\|_{L^\infty(\mathbb{R}^3\times [0,T])} = \lim \limits_{j} \sup\limits_{t\in [0,T]}\|w_{k_j}(\cdot,t)\|_{L^1(\mathbb{R}^3)} = 0\;\;\forall\;T<T_0.
	%  \end{equation*}
	%  This guarantees that $\lim\limits_{j\to\infty} f_{k_j}^\delta = \lim\limits_{j \to\infty} f_{k_j-1}^\delta$, and
     % \begin{equation*}
     %   \|A[f_{k_j}(\cdot,t)]-A[f_{k_j-1}(\cdot,t)]\|_{L^\infty(\mathbb{R}^3)}\leq \|w_{k_j}(\cdot,t)\|_{L^\infty(\mathbb{R}^3)}+\|w_{k_j}(\cdot,t)\|_{L^1(\mathbb{R}^3)} \to 0.	
	%  \end{equation*}			
      Therefore,
      \begin{equation*}
		   \lim \limits_{j\to\infty} \partial_t f_{k_j}^{(\delta)} = \partial_t f^{(\delta)},\;\; \lim\limits_{j\to \infty} Q(f_{k_j-1}^{(\delta)},f_{k_j}^{(\delta)}) =  Q(f^{(\delta)},f^{(\delta)}),
	  \end{equation*}
      and follows that $f^{(\delta)}$ is a solution of the initial value problem
      \begin{equation*}
		    \partial_t f^{(\delta)} = \delta \Delta f^{(\delta)}+ Q(f^{(\delta)},f^{(\delta)}) \text{ in } \mathbb{R}^3\times [0,T_1),\;\;f^{(\delta)}(\cdot,0) = \fin.
	  \end{equation*}
      That proves the first assertion of the theorem. \\
            As for the second one, suppose that
      \begin{equation*}
            \limsup\limits_{t \to T_1^-} \|f^{(\delta)}(\cdot,t)\|_{L^\infty(\mathbb{R}^3)} \leq C,
      \end{equation*}		     
      for some finite $C$. This implies that there exists $0<\rho\le \frac{1}{4C}$ such that %, for any $\rho>0$ there is a large enough $j_0=j_0(\rho)$ such that,
      \begin{equation*}
          f^{(\delta)}(0,T_1-\rho) \leq 2C.
      \end{equation*}  
      Hence one can consider the linear problem 
      $$
        \partial_t f  = \delta \Delta f+Q(f_{k-1}^{(\delta)},f), \quad x\in \R^3, \quad t>T_1-\rho,
        $$
      subject to initial conditions $f_k^\delta (\cdot,T_1-\rho)= f^{(\delta)}(\cdot,T_1-\rho)$. From Lemma \ref{lem:short time} we know that
      $$
      f_k^\delta (0,t)\le \frac{f^{(\delta)}(0,T_1-\rho)}{1-f^{(\delta)}(0,T_1-\rho)(t-T_1+\rho)},
      $$
      which implies $f_k^\delta(0,t)$ is bounded (among others) at the time $\bar{t}= \frac{1}{2 f^{(\delta)}(0,T_1-\rho)} + \varepsilon$, for some small $\varepsilon >0$. The fact that $f^{(\delta)}(0,T_1-\rho) \le 2C$ and $\rho \le \frac{1}{4C}$ implies 
      $$
      \bar{t} \ge \frac{1}{4C} + T_1-\rho +\varepsilon  \ge T_1 + \varepsilon,
      $$
       which contradicts the definition of $T_1$, unless $T_1 =+\infty$. 
      
	%  This $\rho$ may be chosen so that $\rho<(4C)^{-1}$, a fact that will be needed below. Applying the argument from Lemma \ref{lem:short time} to $\tilde f_k^{(\delta)}(0,t) := f_k^{(\delta)}(0,t+T_0-\rho)$ with $k>k_{j_0}$, it follows that
   %   \begin{align*}
   %       f_k^{(\delta)}(0,t+T_0-\rho) & \leq \frac{2C}{1-2C t},\;\;\forall\; t\in [0,(2C)^{-1}), k>k_{j_0}
   %   \end{align*}
   %   Note that $t<(2C)^{-1}$ since $\rho<(4C)^{-1}$, then, $t=2\rho$ above leads to
   %   \begin{align*}
   %       f_k^{(\delta)}(0,T_0+\rho) \leq \frac{2C}{1-4C \rho},\;\;\forall\;k>k_{j_0}.
    %  \end{align*}
   %   Since $\rho<(4C)^{-1}$, it follows that
   %%   \begin{equation*}
    %       \limsup \limits_{k} \|f_k^{(\delta)}(\cdot,T_0+\rho)\|_{L^\infty(\mathbb{R}^3)}<\infty,
    %  \end{equation*}
      %in contradiction with the definition of $T_0$. 
    %  The above arguments shows that 
    %  \begin{equation*}
    %        \limsup\limits_{t \to T_0^-} \|f^{(\delta)}(\cdot,t)\|_{L^\infty(\mathbb{R}^3)} = \infty,
    %  \end{equation*}		     
     % The theorem is proved.	       
	\end{proof}

    \begin{thm}\label{thm: short time existence}
	  Let $\{f^{(\delta)}\}_{\delta>0}$ be as given by Lemma \ref{lem: short time existence regularized}, with either $Q=\QL$ or $Q=\QKS$. There exists a strictly positive and possibly infinite $T_0$, such that
      \begin{equation*}
        \textnormal{(along a subsequence) }\; f^{(\delta)} \to f \textnormal{ in } C^{2 ; 1}_{\textnormal{loc}}\left ( \mathbb{R}^3\times [0,T_0)\right ),
      \end{equation*}
      where $f(v,0)=\fin(v)$ and $f$ solves \eqref{eqn: Landau} or \eqref{eqn: Krieger Strain}, depending on $Q$.  Moreover, $T_0 = \infty$ or 
      \begin{equation*}
        \lim \limits_{t \to T_0^-} \|f(\cdot,t)\|_{L^\infty(\mathbb{R}^3)} = \infty,
      \end{equation*}
      and in either case we have the estimate $\frac{1}{2\fin(0)}<T_0$.
	\end{thm}

    \begin{proof}
	   Let
      \begin{equation*}
        T_0 := \{ t\geq 0 \mid \limsup\limits_{\delta} \|f^{(\delta)}\|_{L^\infty(\mathbb{R}^3\times [0,t])} < \infty \}.
      \end{equation*}
      Clearly $T_0> \frac{1}{2\fin(0)}$ and possibly $T_0=+\infty$. In case $T_0 <\infty$ 
      \begin{equation*}
          \left \{ \begin{array}{rl}          
           \limsup\limits_{\delta \to 0^+} \|f^{(\delta)}\|_{L^\infty(\mathbb{R}^3\times [0,t])} & = +\infty,\;\;\forall\; t>T_0,\\
           \limsup\limits_{\delta \to 0^+} \|f^{(\delta)}\|_{L^\infty(\mathbb{R}^3\times [0,t])} & < +\infty,\;\;\forall\; t<T_0.
          \end{array}\right.
      \end{equation*}      		
      The local uniform convergence of the $f^{(\delta)}$ and its derivatives also guarantees that
      \begin{equation*}
		    \partial_t f^{(\delta)} \to \partial_t f,\;\; Q(f^{(\delta)},f^{(\delta)}) \to Q(f,f),
	  \end{equation*}
      uniformly in compact subsets of $\mathbb{R}^3\times [0,T_0)$. In conclusion, 
      \begin{equation*}
		    \partial_t f = Q(f,f) \text{ in } \mathbb{R}^3\times (0,T_0),\;\;f(\cdot,0) = \fin.
	  \end{equation*}   
	  	   As for the second one, the proof mimics the one in Lemma \ref{lem: short time existence regularized}. 
	   Suppose that
      \begin{equation*}
            \limsup\limits_{t \to T_0^-} \|f(\cdot,t)\|_{L^\infty(\mathbb{R}^3)} \leq C <+\infty,
      \end{equation*}		     
      for some finite $C$. This implies that there exists $0<\rho\le \frac{1}{4C}$ such that %, for any $\rho>0$ there is a large enough $j_0=j_0(\rho)$ such that,
      \begin{equation*}
          f(0,T_0-\rho) \leq 2C.
      \end{equation*}  
      Hence one can consider the problem 
      $$
        \partial_t f  = \delta \Delta f+Q(f,f), \quad x\in \R^3, \quad t>T_0-\rho,
        $$
      subject to initial conditions $f^\delta (\cdot,T_0-\rho)= f(\cdot,T_0-\rho)$. A similar argument as in  Lemma \ref{lem:short time} implies that 
            $$
      f^\delta (0,t)\le \frac{f(0,T_0-\rho)}{1-f(0,T_0-\rho)(t-T_0+\rho)},
      $$
      which implies $f^\delta(0,t)$ is bounded (among others) at the time $\bar{t}= \frac{1}{2 f(0,T_0-\rho)} + \varepsilon$, for some small $\varepsilon >0$. The fact that $f(0,T_0-\rho) \le 2C$ and $\rho \le \frac{1}{4C}$ implies 
      $$
      \bar{t} \ge \frac{1}{4C} + T_0-\delta +\varepsilon  \ge T_0+ \varepsilon,
      $$
       which contradicts the definition of $T_0$, unless $T_0 =+\infty$. 
    	
    \end{proof}
  
    \begin{rems}
    Since all the above estimates solely depend on the $L^\infty(\R^3)$- and $L^1(\R^3)$-norm of the initial data, short time existence for \eqref{eqn: Landau} and \eqref{eqn: Krieger Strain} for initial data as in \eqref{eqn: initial assumptions} can be obtained from Theorem \ref{thm: short time existence} via a simple limit argument. 
    \end{rems}

   \section{Pointwise bounds and proof of Theorem \ref{thm: conditions for Landau blowup}}\label{sec: contact analysis}

\subsection{Conditional pointwise bound} The following lemma is the key argument for the proofs of Theorem \ref{thm: conditions for Landau blowup} and Theorem \ref{thm: global existence KS}. It consists of a barrier argument (this is where the radial symmetry and monotonicity is needed) and affords control of certain spatial $L^p$-norms of the solution.

We first recall that any solution to equation (\ref{eqn: Krieger Strain alpha}) or (\ref{eqn: Landau}) preserves its mass over time. Moreover any solution to (\ref{eqn: Landau}) preserves its energy over time, i.e. $\int_{\R^3}|v|^2 f(v,t)\;dv = \int_{\R^3}|v|^2 f_{in}(v)\;dv=:E_{in}$.

    \begin{lem}\label{lem: propagation of Lp bounds}
        Suppose $f:\mathbb{R}^3\times [0,T] \to\mathbb{R}_+$ is a classical solution of \eqref{eqn: approximate Linear Cauchy problem}. Suppose there exists a modulus of continuity $\omega(r)$, some $R_0>0$ and $\delta>0$ such that 
	 
	 \begin{equation}\label{eqn: uniform integrability f over v}
       \sup \limits_{t\in[0,T]} \left (\int_{B_r}g(w,t)^{3/2}\;dw \right )^{2/3}\leq \omega(r),\;\;\forall\; 0<r \le R_0,
     \end{equation}
        
    \begin{equation}\label{eqn: lower bound on a}
         	  a[g] \geq \delta,\;\;\forall\; 0<|v|\leq R_0,\;t\in [0,T].
    \end{equation}
         Then, for any $0<\gamma  <1$, there is a $r_0 = r_0(\delta,\omega(\cdot),\gamma)$, $0<r_0< R_0$ such that
         \begin{equation}\label{prop_L6_norm}
         	  f(v,t) \leq  \max \left \{ \tfrac{3}{4\pi}r_0^{\gamma-3}\|f\|_{L^1(\mathbb{R}^3)}, (\tfrac{3}{4\pi})^{{\gamma}/{3}}\|\fin\|_{L^{3/\gamma}_{\textnormal{weak}}}\; \right \} |v|^{-\gamma},\; \textnormal{ in } B_{r_0}\times [0,T].
         \end{equation}
    \end{lem}
    
    \begin{rem}
         It is easy to see that for any radially decreasing function $h(v)$ the condition that $h$ belongs to $L^p_w(\R^3)$ implies that $h$ lies below a power function of the form $1/{|v|^{ 3/p}}$, and viceversa. More precisely,
         \begin{equation}\label{weak_p}
         	  \|h(v)\|_{L^p_\textnormal{weak}} \leq C \Leftrightarrow h(v) \leq C \left ( \tfrac{3}{4\pi}\right )^{{1}/{p}} |v|^{-3/p}.
         \end{equation}        
    \end{rem}
    
    \begin{proof}
         Let $U_\gamma$ be as in Lemma \ref{lem: classical radial barriers}. We first show the existence of some $r_0>0$ such that
         \begin{equation}\label{Q}
         	  Q(g,U_\gamma) \leq 0 \text{ in } B_{r_0} \times [0,T].
         \end{equation}
         According to \eqref{eqn: lower bound on a},
         \begin{equation}\label{a_unif_ellip}
         	  a[g]|v|^{-2} \geq  \delta |v|^{-2},\;\;\forall\;|v|\in [0,R_0],\; t\in [0,T].
         \end{equation}
         On the other hand, since $g$ is radially decreasing
         \begin{equation*}
         	  g(v,t) \leq  \frac{3}{4\pi |v|^3}\int_{B_{|v|}}g(w,t)\;dw. 
         \end{equation*}
          For all  $|v|>0,\;t\in [0,T]$ H\"older's inequality and \eqref{eqn: uniform integrability f over v} yield 
         \begin{align*}
         	  g(v,t) & \leq  \frac{3}{4\pi |v|^3}\left (\int_{B_{|v|}}g(w,t)^{3/2}\;dw \right )^{2/3}\left (\frac{4\pi}{3}|v|^3 \right)^{1/3}\\
                     &  = \left (\frac{3}{4\pi}\right )^{2/3}\left (\int_{B_{|v|}}g(w,t)^{3/2}\;dw \right )^{2/3}\frac{1}{|v|^2}\leq \left (\frac{3}{4\pi}\right )^{2/3} \omega(|v|)|v|^{-2}.%\;\;\forall\; |v|>0,\;t\in [0,T].
         \end{align*}
         %Chose now $r_0 \in (0,R_0]$ small enough so that
        % \begin{equation*}
         %	    \left (\tfrac{3}{4\pi}\right )^{2/3} \omega(r_0)\leq \tfrac{1}{3 \alpha} \gamma(1-\gamma)\delta,
        % \end{equation*}
        % with $\gamma \in (0,1)$. 
       %  Then, 
       %  \begin{equation*}
        % 	  g(v,t) \leq \tfrac{1}{3 \alpha} \gamma(1-\gamma) a[g]|v|^{-2},\;\;\forall\;|v|\in [0,r_0],\; t\in [0,T].
       %  \end{equation*}
        Therefore in $B_{r_o}$ it holds,
         \begin{align*}
         	  \QL(g,U_\gamma) ,\quad \QKS (g,U_\gamma) & \le U_\gamma \left ( -\tfrac{1}{3}\gamma(1-\gamma) a[g]|v|^{-2}+ g \right ) \\
	 & \leq U_\gamma |v|^{-2}\left ( -\tfrac{1}{3}\gamma(1-\gamma) \delta +\left (\frac{3}{4\pi}\right )^{2/3} \omega(r_0) \right ).
	 %\QKS(g,U_\gamma) & = U_\gamma \left ( g-\gamma(1-\gamma) a[g]|v|^{-2}\right ) 
	  \end{align*}
	  Hence \eqref{Q} holds by choosing $\gamma \in (0,1)$  and $r_0 \in (0,R_0]$ small enough so that
         \begin{equation*}
         	    \left (\tfrac{3}{4\pi}\right )^{2/3} \omega(r_0)\leq \tfrac{1}{3} \gamma(1-\gamma)\delta.
         \end{equation*}
      Since $f(v,t)$ is also radially decreasing, it holds
         \begin{equation}\label{rad}
         	  f(r,t) \leq \tfrac{3}{4\pi |v|^3}\|f\|_{L^1(\mathbb{R}^3)}.
         \end{equation}
         The function $\tilde U_\gamma(v)$ defined as 
         \begin{equation*}
         	  \tilde U_\gamma(v) := \max \left \{ \tfrac{3}{4\pi}r_0^{\gamma-3}\|f\|_{L^1(\mathbb{R}^3)}, (\tfrac{3}{4\pi})^{{\gamma}/{3}}\|\fin\|_{L^{3/\gamma}_{\textnormal{weak}}}\; \right \}|v|^{-\gamma},
         \end{equation*}
         is a supersolution for the equation solved by $f$ in $B_{r_0}\times [0,T]$, namely
         \begin{equation*}
         	  \left \{ \begin{array}{rll}
         	  	   Q(g,\tilde U_\gamma) & \leq 0 & \textnormal{ in } B_{r_0} \times [0,T],\\
                              \tilde U_\gamma  & \geq f & \textnormal{ on } \partial B_{r_0} \times [0,T],\\
                              \tilde U_\gamma  & \geq f & \textnormal{ on } B_{r_0} \times \{0\}.
         	  \end{array}\right.
         \end{equation*}
         By the comparison principle, \eqref{weak_p} and \eqref{rad} it follows that $f \leq \tilde U_\gamma$ in $B_{r_0} \times [0,T]$. 
    \end{proof}
 
 The next lemma shows that the mass of any radial symmetric solution to \eqref{eqn: Landau} or \eqref{eqn: Krieger Strain} in a compact set can be controlled from below by a constant that only depends on the initial data. More precisely:
 
  \begin{lem}\label{lem: mass lower bound}
      For $f$ solving \eqref{eqn: Landau}, there exists a constant $R>0$ that only depends on $M_{in}$ and $E_{in}$ such that
      \begin{equation}\label{eqn: lower bound on a_proof_L}
        \int_{B_R} f(v,t)\;dv \geq M_{in}/2,\;\; \quad t>0.%\forall\; |v|\le R\geq \sqrt{6}.
      \end{equation}
      For $f$ solving \eqref{eqn: Krieger Strain}, and any radii $R>r>0$ there are $\beta>0$ and $C_0>0$ such that
      \begin{equation}\label{eqn: lower bound on a_proof_nonlinear}
        \int_{B_R \setminus B_r} f(v,t)\;dv \geq C_0e^{-\beta t}\int_{B_{4R}\setminus B_{r/4}} \fin(v)\;dv \quad t>0.			 
      \end{equation}
    \end{lem}

\begin{proof} 
	If $f$ solves \eqref{eqn: Landau}, then
	\begin{equation*}
	   \int_{B_R(0)^c}f(v,t)\;dv \leq R^{-2}\int_{B_R(0)^c}f(v,t)|v|^2\;dv \le  R^{-2}E_{in} .  
	\end{equation*}
	Thus
	\begin{equation*}
      \int_{B_R(0)}f(v,t)\;dv = M(\fin)-\int_{B_R(0)^c}f(v,t)\;dv \geq M_{in}-R^{-2}E_{in}.	
	\end{equation*}
	Estimate (\ref{eqn: lower bound on a_proof_L}) follows by choosing $R$ big enough. 	The corresponding estimate (\ref{eqn: lower bound on a_proof_nonlinear}) for $f$ solving \eqref{eqn: Krieger Strain} can be found in the Appendix. 
 
\end{proof}

    The next lemma says that any solution $f$ to \eqref{eqn: Landau} or \eqref{eqn: Krieger Strain} is a bounded function for all times provided $f$ satisfies \eqref{eqn: uniform integrability f over v}.

\begin{lem}(From $L^{3/2+}$ to $L^\infty$.)\label{lem: L3/2 norm controls blow up}
    	 Let $f: \mathbb{R}^3\times [0,T]\to\mathbb{R}$ be a radially symmetric, radially decreasing solution to \eqref{eqn: initial assumptions}, \eqref{eqn: Landau} and \eqref{eqn: Krieger Strain}, and such that for any modulus of continuity $\omega(r)$ the following estimate holds:
	  \begin{equation}\label{eqn: uniform integrability f over v_nonlinear}
       \sup \limits_{t\in[0,T]} \left (\int_{B_r}f(w,t)^{3/2}\;dw \right )^{2/3}\leq \omega(r),\;\;\forall\; 0<r \le R_0.
     \end{equation}
 Then there exists a constant $C_0>0$ that only depends on $f_{in}$ such that

         \begin{equation}\label{Linfty-bound_alpha}
         	  \sup \limits_{t\in[0,T]} \|f(\cdot,t)\|_{L^\infty(\mathbb{R}^3)} < C_0.
         \end{equation}
    \end{lem}
    
    \begin{proof}
         Proposition \ref{prop: A is C^1,alpha and strictly positive} yields the inequality,
         \begin{equation*}
               A[f](v,t) \geq \lambda \Id,\;\;\forall\; (v,t) \in Q,
         \end{equation*}         
         where $\lambda = \lambda(\fin,T)$. On the other hand, since $f$ is radially decreasing,
         \begin{equation*}
              f(v,t)\leq \tfrac{3}{4\pi} \|f\|_{L^1}  r_0^{-3}\leq \tfrac{3}{4\pi} \|f_{in}\|_{L^1}  r_0^{-3} ,\;\;\forall\;v \notin B_r .
         \end{equation*}
         Thus,
         \begin{equation*}
              f(v,t)\leq \tfrac{3}{4\pi} \|f\|_{L^1}  r_0^{-3}\leq \|\fin\|_{L^\infty}+ \tfrac{3}{4\pi} \|f_{in}\|_{L^1}  r_0^{-3},\;\;\textnormal{ on } \partial_p Q.
         \end{equation*}             
         We now apply Lemma \ref{lem: propagation of Lp bounds} to $f(v,t)$: first note that  (\ref{eqn: lower bound on a}) is a consequence of Lemma \ref{lem: mass lower bound}. Thanks to the uniform integrability of $f^{3/2}$ (\ref{eqn: uniform integrability f over v_nonlinear}), Lemma \ref{lem: propagation of Lp bounds} (for some $\gamma<1/2$) yields 
     \begin{equation*}
             \sup \limits_{t \in [0,T] } \|f(t)\|_{L^6(B_{r_0})} \leq \max \{ r_0^{-5\gamma}. \|\fin\|_{L^1(\mathbb{R}^3)},\|\fin\|_{L^p_{\textnormal{weak}}}  \},\; \textrm{ for some }\;p>6.
         \end{equation*}   
         It follows that $\|f(\cdot,t)\|_{L^6(\mathbb{R}^3)}$ and $\|\nabla a[f(\cdot,t)]\|_{L^\infty(\mathbb{R}^3)}$ are bounded for $t\leq T$. In particular,
         \begin{equation*}
              \|f \nabla a[f]\|_{L^2(\mathbb{R}^3)}<\infty.
         \end{equation*}
         Applying \eqref{eqn: Stampacchia estimate} from Theorem \ref{thm: Stampacchia} we arrive at
         \begin{equation*}
              \|f\|_{L^\infty(Q)}\leq \|f\|_{L^\infty(\partial_p Q)}+C\|f \nabla a[f]\|_{L^{2}(Q)} < \infty,
         \end{equation*}
         which proves the lemma.
    \end{proof}

{\em Proof of Theorem \ref{thm: conditions for Landau blowup}.} Lemma \ref{lem: L3/2 norm controls blow up} guarantees that if (\ref{eqn: uniform integrability f over v_nonlinear}) holds, then any solution to \eqref{eqn: Landau} and \eqref{eqn: Krieger Strain} stays uniformly bounded in time.

	%%%%%%%%%%%%%%%%%%%%%%%%%%%%%%%%%%%	
	%%%%%%%%%%%%%%%%%%%%%%%%%%%%%%%%%%%	
	\section{Mass comparison and proof of Theorem \ref{thm: global existence KS}}\label{mass_comparison}

  In this section we apply the ideas from previous sections to construct global solutions (in the radial, monotone case) for equation \eqref{eqn: Krieger Strain}, namely
  \begin{equation*}
      \partial_t f = a[f]\Delta f + f^2.
  \end{equation*}
  %The local existence theory, as well as the proof of Lemma \ref{lem: L3/2 norm controls blow up} and Lemma \ref{lem: propagation of Lp bounds} for \eqref{eqn: Krieger Strain} are entirely analogous to these for \eqref{eqn: Landau}, and we shall state the results without providing the details. 
  
 % \begin{lem}\label{lem: L3/2 norm controls blow up for KS}
  %  	 Let $f: \mathbb{R}^3\times [0,T]\to\mathbb{R}$ be a radially symmetric, radially decreasing solution to (\ref{eqn: initial assumptions}), (\ref{eqn: Krieger Strain}) and such that 
%		 \begin{equation*}
%		   a[f]\geq \delta \textnormal{ for } r\leq R_0,t\in[0,T].	
%		 \end{equation*}
   %      Then there exists a constant $C_0$ that only depends on $\fin, \delta,R_0$ such that  
   %      \begin{equation}\label{Linfty-bound_alpha KS}
   %      	  \sup \limits_{t\in[0,T]} \|f(\cdot,t)\|_{L^\infty(\mathbb{R}^3)} < C_0.
   %      \end{equation}
   % \end{lem}
  
In view of  Lemma \ref{lem: L3/2 norm controls blow up}, the fact that $T_0=\infty$ in Theorem \ref{thm: conditions for Landau blowup} results from a bound of any $L^p(\R^3)$-norm of $f$, with $p>3/2$. 
For \eqref{eqn: Krieger Strain} the bound of any $L^p(\R^3)$-norm of $f$, with $p>3/2$ will be proven by a barrier argument done at the level of the \emph{mass function} of $f(v,t)$, which is defined by
  \begin{equation*}
    M_f(r,t) = \int_{B_r} f(v,t)\;dv,\;\;\;(r,t)\in \mathbb{R}_+\times (0,T_0).
  \end{equation*}
  Depending on which problem $f$ solves, the associated function $M_f(r,t)$ solves a one-dimensional parabolic equation with diffusivity given by $A^*[f]$ or $a[f]$. 
  \begin{prop}\label{Prop:mass}
    Let $f$ be a solution of \eqref{eqn: Landau} (resp. \eqref{eqn: Krieger Strain}) in $\mathbb{R}^3\times [0,T_0]$, then $M(r,t)$ solves 
    \begin{equation}\label{2}
      \partial_t M_f = A^* \partial_{rr}M_f+\frac{2}{r}\left (\frac{M_f}{8\pi r}-A^* \right )\partial_r M_f\;\;\textnormal{ in } \mathbb{R}_+\times (0,T_0)
    \end{equation}
    \begin{equation}\label{1}
      \left (\textnormal{resp. }\; \partial_t M_f = a \partial_{rr}M_f+\frac{2}{r}\left (\frac{M_f}{8\pi r}-a \right )\partial_r M_f\;\;\textnormal{ in } \mathbb{R}_+\times (0,T_0) \right ).
    \end{equation}
  \end{prop}
  \begin{proof}
    We briefly show how to obtain (\ref{1}); for (\ref{2}) calculations are identical. Using the divergence theorem and the divergence expression in \eqref{eqn: Krieger Strain} we get
    \begin{equation*}
      \partial_t M_f = \int_{\partial_{B_r}} \left (a[f]\nabla f-f\nabla a[f],n\right )\;d\sigma = 4\pi r^2 \left ( a[f] \partial_r f - f \partial_r a[f] \right ).
    \end{equation*}
    Furthermore, straightforward differentiation yields the formulas
    \begin{equation*}
      4\pi r^2 \partial_r f = r^2\partial_r \left (r^{-2}\partial_r M_f \right ),\;\;\;\partial_r a[f] = -(4\pi r^2)^{-1}M_f.
    \end{equation*}
    Substituting these in the expression for $\partial_tM_f$ above we get
    \begin{equation*}
      \partial_t M_f = a[f] r^2\partial_r \left ( \frac{1}{r^2}\partial_r M_f \right )+ \frac{1}{4\pi r^2}M_f \partial_r M_f.
    \end{equation*}
    Expansion and  rearrangement  of the terms result in:
    \begin{align*}
      \partial_t M_f & = a \left ( -\frac{2}{r}\partial_r M_f+\partial_{rr}M_f \right )+ \frac{M_f}{4\pi r^2} \partial_r M_f\\
        & = a \partial_{rr}M_f + \frac{2}{r}\left ( \frac{M_f}{8\pi r} - a \right )\partial_r M_f,
    \end{align*}
and the thesis follows.
  \end{proof}
  Define the linear parabolic operator $L$ in $\mathbb{R}_+\times (0,T)$ as
  \begin{equation*}
      Lh := \partial_t h - a \partial_{rr}h-\frac{2}{r}\left (\frac{M_f}{8\pi r}-a[f] \right )\partial_r h.
  \end{equation*}
  The above proposition simply says that $L M_f = 0$ in $\mathbb{R}_+\times (0,T)$. The next proposition identifies suitable supersolutions for $L$.
  \begin{prop}\label{prop:supersol KS}
    If $m\in [0,2]$ and $h(r,t) = r^m$ then $Lh\geq 0 \textnormal{ in } \mathbb{R}_+\times (0,T)$.
  \end{prop}
  
  \begin{proof}
    By direct computation we see that
    \begin{equation*}  
      Lh = - m r^{m-2}\left [ (m-1)a + 2\left ( \frac{M_f}{8\pi r} - a[f]\right ) \right ]. 
    \end{equation*}
    On the other hand,
    \begin{equation*}
      a[f](r) = \frac{1}{4\pi r}\int_{B_r}f\;dv+ \int_{B_r^c}\frac{f}{4\pi |v|}\;dv \geq \frac{M_f}{4\pi r},
    \end{equation*}
    which guarantees that $\tfrac{1}{2}a[f](r) \geq \frac{M_f}{8\pi r}$. Thus,
    \begin{align*}  
      Lh & = m r^{m-2}\left [ (1-m)a[f] + 2\left (a[f]- \frac{M_f}{8\pi r} \right ) \right ],\\
         %& \geq m r^{m-2}\left [ (1-m)a + 2\left (\tfrac{1}{2}a\right ) \right ],\\
         & \geq m r^{m-2}( (2-m)a[f] \geq 0. 
    \end{align*}
    The last inequality being true for $m\leq 2$.

  \end{proof}

  \begin{proof}[Proof of Theorem \ref{thm: global existence KS}]
      Since $\fin \in L^p$ for some $p>3/2$, there is some $\alpha \in (0,1)$ and some $C_0>0$ such that
    \begin{equation*}
        M_{f_{in}}(r,0) = \int_{B_r}f_{in}\;dv\leq C_0r^{1+\alpha}.
    \end{equation*} 
    Moreover, since $f(\cdot,t)$ has total mass $1$ for every $t>0$, we also have
    \begin{equation*}
        M_f(r,t)\leq 1,\;\;\forall\;r>0,t\in(0,T). 
    \end{equation*} 
    
    Proposition \ref{prop:supersol KS} says that $h=Cr^{1+\alpha}$ is a supersolution of the parabolic equation solved by $M_f$ in $\mathbb{R}_+\times (0,T)$. Then choosing $C:=\max\{C_0,1\}$ comparison principle yields 
    \begin{equation} \label{4}
      M_f(r,t)\leq h(r) = Cr^{1+\alpha}\;\;\textnormal{ for} \;r\in (0,1),\; t\in (0,T).      
    \end{equation}

    %Let $C=\max\{C_0,1\}$ and $h(r)=Cr^{1+\alpha}$. Proposition \ref{prop:supersol KS} says that $h$ is a supersolution of the parabolic equation solved by $M_f$ in $\mathbb{R}_+\times (0,T)$. Moreover, by construction it is clear that
  %  \begin{equation*}
  %    M_f(r,t)\leq h(r),\;\;\textnormal{ on } \partial_p\left[(0,1)\times (0,T)\right ],
  %  \end{equation*}
 %   that is, on $\partial_p\left[(0,1)\times (0,T)\right ] = \{0,1\} \times (0,T) \cup (0,1) \times \{0\}$. Then, by the comparison principle
 %   \begin{equation} \label{4}
 %     M_f(r,t)\leq h(r) = Cr^{1+\alpha}\;\;\textnormal{ in } (0,1)\times (0,T).      
 %   \end{equation}
    Since $f(v,t)$ is  readily symmetric and decreasing, bound (\ref{4}) implies that $f(|v|,t)\le \frac{3C}{4\pi}\frac{1}{|v|^{2-\alpha}}$ for $v\in B_1$; hence there is some $p'>3/2$ and some $C_{p'}>0$ such that
    \begin{equation*}
      \|f(\cdot,t)\|_{L^{p'}(B_1)}\leq C_{p'},\;\;\forall\;t\in(0,T).
    \end{equation*}

        Then Lemma \ref{lem: propagation of Lp bounds} says that $f(v,t)$ is bounded in $\mathbb{R}^3 \times (0,T_0)$. By Lemma \ref{lem: L3/2 norm controls blow up}, it follows $T_0=+\infty$ and we have a global in time smooth solution.
         
 %   Then (the analogue of) Lemma \ref{lem: L3/2 norm controls blow up} says that $f(v,t)$ is bounded in $\mathbb{R}^3 \times (0,T_0)$. By Theorem \ref{thm: short time existence}, it follows $T_0=+\infty$ and we have a global smooth solution.
    
  \end{proof} 
       
As a corollary of Theorem \ref{thm: global existence KS} and Proposition \ref{Prop:mass} we give another criterium under which blow-up for the classical Landau equation is ruled out:
  \begin{cor}
    Suppose that for all $t \in [0,T_0]$ there is some $r_0>0$ and $0<\lambda< 8\pi $ such that
	\begin{equation*}
      M_f(r,t)\leq \lambda r A^*(r,t) \;\;\forall\;\; r<r_0.
	\end{equation*}
    Then any solution to (\ref{eqn: Landau}) is bounded for any $t>0$.
    %$\|f(\cdot,t)\|_{L^\infty(\mathbb{R}^3)}$ cannot blow up as $t \to T_0^-$.
  \end{cor}   
	   
\section{Appendix}

  \begin{proof}[Proof of Proposition \ref{prop: preservation of radiality}] The radial symmetry of any solution $f$ to \eqref{eqn: approximate Linear Cauchy problem} follows by the uniqueness property of \eqref{eqn: approximate Linear Cauchy problem} and by the fact that $Q(g,f)$ commutes with rotations, as shown below. We first rewrite the collision operator as 
  \begin{equation*}
    Q(g,f)= \dive(A[g]\nabla f-f\nabla a[g])-(1-\alpha)fg = a[g]\Delta f -  \dive(\tilde{A}[g]\nabla f) + \alpha f g ,
  \end{equation*}
    with $ \tilde{A}[g]\nabla f := \int \frac{g(|v-y|)}{|y|^3} \langle \nabla f(v),y\rangle y \;dy$. 
    
    Let $\mathbb{T}$ be a rotation operator. Since $g$ is radially symmetric, so is $a[g]$. Hence   
  \begin{equation*}
    a[g]\Delta (f\circ \mathbb{T}) = a[g\circ  \mathbb{T}]\Delta (f\circ \mathbb{T}) = (a[g]\circ  \mathbb{T})(\Delta f\circ \mathbb{T})=(a[g] \Delta f ) \circ \mathbb{T},
  \end{equation*}
  taking into account that the Laplacian operator commutes with rotations. Moreover 
    
   \begin{align*}
   \dive ( \tilde{A}[g]\nabla f(\mathbb{T}v)) &= \dive \left(  \int \frac{g(|v-y|)}{|y|^3} \langle \nabla f(\mathbb{T}v),y\rangle y \;dy \right) \\
   &=  \dive \left(  \int \frac{g(|v-y|)}{|y|^3} \langle \mathbb{T}^* \nabla_z f(z)_{|_{z=\mathbb{T}v}} ,y\rangle y \;dy \right)  \\
   & =  \dive \left(  \int \frac{g(|\mathbb{T}(v-y)|)}{|y|^3} \langle  \nabla_z f(z)_{|_{z=\mathbb{T}v}} ,\mathbb{T}y\rangle \mathbb{T}^*\mathbb{T}y \;dy \right)  \\
  & = \dive \left( \mathbb{T}^* \underbrace{\int \frac{g(|\mathbb{T}v-y)|)}{|y|^3} \langle  \nabla_z f(z)_{|_{z=\mathbb{T}v}} ,y\rangle y \;dy}_{=:V(\mathbb{T}v)} \right)  \\
  &= Tr(\mathbb{T}^* Jac(V)_{|_{z=\mathbb{T}v}} \mathbb{T}) + \underbrace{\nabla (Tr(\mathbb{T}^*))}_{=0}\cdot V(\mathbb{T}v)\\
   &= Tr(\mathbb{T}\mathbb{T}^* Jac(V)_{|_{z=\mathbb{T}v}} ) \\
  &= Tr({\mathbb{I}d} \;Jac(V)_{|_{z=\mathbb{T}v}} ) \\
  &=  \dive \left(  \int \frac{g(|z-y|)}{|y|^3} \langle  \nabla_z f(z) ,y\rangle y \;dy \right)  \circ \mathbb{T}.
    \end{align*}
    Hence $Q(g,f(\mathbb{T}v)) = Q(g,f)  \circ \mathbb{T}$.\\
    Now we rewrite the linear equation (\ref{eqn: approximate Linear Cauchy problem}) in spherical coordinates:
    \begin{equation}\label{eq_rad_symm}
              \partial_t f = A^*\partial_{rr}f +\tfrac{a-A^*}{r}\partial_rf + fg,
        \end{equation}
        with  $A^*[g](v) := (A[g](v)\hat v,\hat v)$, $ \hat v := \frac{v}{|v|}$ and differentiate (\ref{eq_rad_symm}) with respect to $r$. The function $w:= \partial_r f$ satisfies the following inequality:
        \begin{equation*}
              \partial_t w \le A^*\partial_{rr}w +\tfrac{a-A^*}{r}\partial_rw + wg+ \partial_r A^* \partial_r w +\partial_r \left (\tfrac{a-A^*}{r}\right) w.
        \end{equation*}
       If  $w(\cdot,0) \leq 0$ it follows from maximum principle that $w(\cdot,t)\leq 0$ for all $t\geq 0$. In other words, the (negative) sign of $\partial_r f$ is preserved in time.
        
    \end{proof}

    \begin{proof}[Proof of Proposition \ref{prop: Landau in radial coordinates}] The identity \eqref{a_symm} is a classical and a proof can be found in \cite{LieLos2001}[Section 9.7]. To prove \eqref{A_symm},  let $v \in \mathbb{R}^3$ non-zero, $r:= |v|$, then
    \begin{equation*}
         (A[g](v)\hat v,\hat v) = \frac{1}{8\pi}\int_{\mathbb{R}^3} \frac{1}{|v-w|}g(w)\left ( (\Id-\tfrac{v-w}{|v-w|}\otimes \tfrac{v-w}{|v-w|})\hat v,\hat v \right )\;dw.
    \end{equation*}
    Note that
    \begin{equation*}
        \left ( (\Id-\tfrac{v-w}{|v-w|}\otimes \tfrac{v-w}{|v-w|})\hat v,\hat v \right ) = 1-\cos(\hat \theta(w))^2,
    \end{equation*}
    where $\hat \theta$ denotes the angle between $w-v$ and $v$. Consider, for $0\leq t,r$, the function
    \begin{equation*}
        I(r,t) := \int_{\partial B_t} \frac{1-\cos(\hat \theta)^2}{|v-w|}\;dw.
    \end{equation*}     
    The function $I(r,t)$ encodes all the information about $A^*$. In particular, integration in spherical coordinates yields the expression
    \begin{equation*}
        A^*[h](v) = \frac{1}{8\pi}\int_0^\infty f(t)I(|v|,t)\;dt.
    \end{equation*}
    As it turns out, $I(r,t)$ has rather different behavior according to whether $r<t$ or not. By averaging in the $v$ variable, it is not hard to see that
    \begin{equation*}
        I(r,t) = \frac{t^2}{r^4}I(t,r),\;\; \forall\; r<t.
    \end{equation*}                   
    Accordingly, we focus on $I(r,t)$ when $r>t$. To do so, denote by $\theta$ the angle between $w$ and $v$ and observe that    
    \begin{equation*}
        1-\cos(\hat \theta)^2  = \sin (\hat \theta)^2 = \frac{t^2-t^2\cos(\theta)^2}{|v-w|^2} =  \frac{t^2-w_1^2}{|v-w|^2},
    \end{equation*}                     
    where $w_1 = (w,\hat v)$. Thus,
    \begin{align*}
        I(r,t) & = \int_{\partial B_t} \frac{t^2-w_1^2}{|v-w|^3}\;dw\\
               & = \int_{\partial B_t} \frac{t^2-w_1^2}{(t^2-w_1^2+(r-w_1)^2)^{3/2}}\;dw\\
               & = \int_{\partial B_t} \frac{t^2-w_1^2}{(t^2-2rw_1+r^2)^{3/2}}\;dw\\               
               & = \int_{\partial B_1} \frac{t^2(1-z_1^2)}{t^3(1-2(\tfrac{r}{t})z_1+(\tfrac{r}{t})^2)^{3/2}}\;t^2dz\\
               & = \int_{\partial B_1} \frac{1-z_1^2}{(1-2(\tfrac{r}{t})z_1+(\tfrac{r}{t})^2)^{3/2}}\;tdz.                                      
    \end{align*}
    This surface integral can be written entirely as an integral in terms of the variable $z_1\in(-1,1)$, 
    \begin{align*}
        I(r,t) & = 2\pi t\int_{-1}^1 \frac{1-z_1^2}{(1-2(\tfrac{r}{t})z_1+(\tfrac{r}{t})^2)^{3/2}}\;dz_1.
    \end{align*}
    For brevity, set for now $s=r/t$, then
    \begin{align*}
        \int_{-1}^1 \frac{1-z_1^2}{(1-2sz_1+s^2)^{3/2}}dz_1 & = \frac{-2s^4+2s^3+2s-2}{3s^3\sqrt{s^2-2s+1}}-\frac{-2s^4-2s^3-2s-2}{3s^3\sqrt{s^2+2s+1}}\\
                                                            & = \frac{-2s^4+2s^3+2s-2}{3s^3(s-1)}-\frac{-2s^4-2s^3-2s-2}{3s^3(s+1)}\\
                                                            & = \frac{-2s^4+2s^3+2s-2}{3s^3(s-1)}+\frac{2s^4+2s^3+2s+2}{3s^3(s+1)}.
    \end{align*}
    Furthermore,
    \begin{align*}
        \frac{-2s^4+2s^3+2s-2}{3s^3(s-1)}+\frac{2s^4+2s^3+2s+2}{3s^3(s+1)} & = \frac{2}{3s^3}\left (\frac{-s^4+s^3+s-1}{s-1}+\frac{s^4+s^3+s+1}{s+1} \right )\\
                                                                           & = \frac{2}{3s^3} \frac{(-s^4+s^3+s-1)(s+1)+(s^4+s^3+s+1)(s-1)}{s^2-1}\\
                                                                           & = \frac{2}{3s^3} \frac{2s^2-2 }{s^2-1} = \frac{4}{3s^3}.
    \end{align*}
    Then, since $s=r/t$, we conclude that
    \begin{align*}
        I(r,t) & = 8\pi\frac{t^4}{3r^3},\;\;\text{ for } t<r,\\
        I(r,t) & = 8\pi\frac{1}{3t},\;\;\;\;\text{ for } t>r.
    \end{align*}
    Going back to $A^*[h]$, the above leads to     
    \begin{align*}
        A^*[h](v) & = \int_0^r h(t)I(r,t)\;dt+\int_r^\infty h(t)I(r,t)\;dt\\
                  & = \frac{1}{3r^3}\int_0^r h(t)t^4\;dt+\frac{1}{3}\int_r^\infty h(t) t\;dt.
    \end{align*}

    \end{proof} 

\begin{proof}[Proof of Lemma \ref{lem: mass lower bound}] 
	
	This argument is inspired by the one in Section 2.6 in \cite{KriStr2012}.  For $\beta,R,r$ (with $0<r<R$, $0<\beta$) consider the function
	\begin{equation*}
	  \Phi(v,t) := e^{-\beta t}(|v|-R)^2(|v|-r)^2.
	\end{equation*}
	Since $\Phi$ is a $C^{1,1}$ function with compact support, it holds
    \begin{align*}
        \frac{d}{dt}\int_{\mathbb{R}^3}f(v,t)\Phi(v)\;dv & = -\int_{\mathbb{R}^3} (a\nabla f-f\nabla a,\nabla \Phi)\;dv\\
                                                                 & = \int_{\mathbb{R}^3}f \dive(a\nabla \Phi)\;dv+\int_{\mathbb{R}^3}f(\nabla a,\nabla \Phi)\;dv.
                                                                     \end{align*}
    Hence
    \begin{align*}
        \dive(a\nabla \Phi)+(\nabla a,\nabla \Phi) & = a\Delta \Phi +2(\nabla a,\nabla \Phi)\\ 
                                                   & = a\Phi''+\frac{2}{|v|}\left (a+|v|a'\right) \Phi'\\
                                                   &=a\Phi''+ \frac{2}{|v|}\Phi'\int_{|v|}^{+\infty}s f(s,t)\;ds.
    \end{align*}
       It holds:
    \begin{align*}
         \Phi'(s) &  = 2(R-s)(s-r)(-(s-r)+R-s) =2(R-s)(s-r)(R+r-2s),\\
         \Phi''(s) & = 2(R-s)(r+R-2s)-2(s-r)(r+R-2s)-4(R-s)(s-r),\\
         \Phi'(r)& =\Phi'(R)=0,\;\;\Phi''(r)=\Phi''(R)=2(R-r)^2,\\
            |\Phi''|,&\;|\Phi'| \leq C_{\delta,r,R}\Phi, \quad |v|\in ((1+\delta)r,(1-\delta)R).
\end{align*}
   
    Hence in a small neighborhood of $|v|=R$ and $|v|=r$ one can show that $ \frac{d}{dt}\int_{\mathbb{R}^3}f(v,t)\Phi(v)\;dv \ge0$; more precisely it holds
    \begin{equation*}
        \dive(a\nabla \Phi)+(\nabla a,\nabla \Phi) \geq 0 \text{ in } B_R\setminus B_{(1-\delta)R} \cup B_{(1+\delta)r}\setminus B_r.
    \end{equation*}
  Since $ a[g](v)  \leq \frac{\|g\|_{L^1(\R^3)}}{|v|}$, it follows
       
    \begin{align*}
        \frac{d}{dt}\int_{\mathbb{R}^3}f(v,t)\Phi(v)\;dv & \geq -C_{\delta,r,R}\frac{\|g\|_{L^1(\R^3}}{r}\int_{B_{(1-\delta)R} \setminus B_{(1+\delta)r}} f(v,t)\Phi(v)\;dv \\%-(1-\alpha)\int_{\mathbb{R}^3}fg\Phi(v)\;dv\\
                                                         & \geq -\frac{\|g\|_{L^1(\R^3)}}{r}C_{\delta,r,R}\int_{\mathbb{R}^3}f(v,t)\Phi(v)\;dv.
    \end{align*}
    This above differential inequality implies %\eqref{eqn: lower bound on a_proof} that
    
    \begin{equation*}
        \int_{\mathbb{R}^3} f(v,t)\Phi(v)\;dv \geq e^{-\beta T}\int_{\mathbb{R}^3}\fin \Phi(v)\;dv,\;\;\forall\;t<T,
    \end{equation*}        
    where $\beta = C_{r,R,\alpha} \|g\|_{L^1} $. Finally, since
    \begin{equation*}
        \Phi(v)\leq \frac{1}{4}(R-r)^2\;\text{ in } B_R\setminus B_r\;,\;\; \Phi(v)\geq \frac{R^2r^2}{4}, 
    \end{equation*}
         we conclude that
    \begin{equation*}
        \int_{B_R\setminus B_r} f(v,t)\;dv \geq \frac{R^2r^2}{(R-r)^4}e^{-\beta T}\int_{B_{R/2}\setminus B_{2r}}\fin(v) \Phi(v)\;dv,\;\;\forall\;t<T.        
    \end{equation*} 
    \end{proof}

%\nocite{DegongLemou97,DesVil2000b,Vil2003,Guo02,KriStr2012,AleLinLia2013,DesVil2005,FouGue2009,Uka1974,Wu2013}

\bibliography{landaurefs}
\bibliographystyle{plain}

\end{document}